\documentclass[12pt,reqno]{amsart}
\usepackage[T1]{fontenc}
\usepackage[english]{babel}
\usepackage[utf8]{inputenc}
\usepackage{mathtools}
\usepackage{amssymb}
\usepackage{csquotes}
\usepackage{enumitem}
\usepackage{braket}

\usepackage{pgfplots}
\usepackage{tikz}
\usetikzlibrary{arrows}
\usetikzlibrary{matrix}
\usetikzlibrary{cd}
\usepackage{graphicx} 
\usepackage{caption} 
\usepackage{subcaption} 
\pgfplotsset{compat=1.18}

\usepackage[a4paper,margin=2.5cm]{geometry}

\usepackage{hyperref}
\hypersetup{
    colorlinks=true,
    linkcolor=blue,
    filecolor=magenta,      
    urlcolor=cyan,
}  

\newcommand{\C}{\mathbb{C}}
\newcommand{\N}{\mathbb{N}}
\newcommand{\norm}[1]{\left\|#1\right\|}
\newcommand{\R}{\mathbb{R}}
\newcommand{\Id}{\operatorname{Id}}
\newcommand{\Hil}{\mathcal{H}}

\newcommand{\CAT}{\operatorname{CAT}(0)}
\newcommand{\RD}{\operatorname{(RD)}}

\numberwithin{equation}{section}

\title[Operator-Valued Haagerup Inequalities]{A Centroid Framework for Operator-Valued Haagerup Inequalities}
\author[P.\ O.\ Santos]{Patrick Oliveira Santos}
\address{Patrick Oliveira Santos. Division of Science, NYU Abu Dhabi, Abu Dhabi, UAE}
\email{po2150@nyu.edu}

\begin{document}
\begin{abstract}
We introduce a centroid-based block decomposition of the left-regular representation and use it to prove operator-valued Haagerup inequalities for finitely generated groups with centroid maps. The decomposition separates the three centroid growth conditions into distinct operator-norm contributions and, under a factorization hypothesis for the associated Schur multipliers, yields complementary lower bounds. For median groups of rank $\nu$, the centroid blocks are related to spherical operators, giving two-sided estimates with the explicit factor $\binom{\nu+r}{r}$; this applies in particular to right-angled Artin groups and to groups acting on finite-rank $\operatorname{CAT}(0)$ cube complexes. For finite-rank coarse median groups, we construct a coarse iterate median with error independent of the cardinality of the input set and obtain spherical-block estimates with factor $o(r^{1+\nu/2+\varepsilon})$ for every $\varepsilon>0$.
\end{abstract}

\maketitle

\newtheorem{theorem}{Theorem}[section]
\theoremstyle{plain}

\newtheorem{corollary}[theorem]{Corollary}
\newtheorem{lemma}[theorem]{Lemma}
\newtheorem{conjecture}[theorem]{Conjecture}
\newtheorem{proposition}[theorem]{Proposition}

\theoremstyle{definition}
\newtheorem{example}[theorem]{Example}
\newtheorem{definition}[theorem]{Definition}
\newtheorem{remark}[theorem]{Remark}
\newtheorem{question}[theorem]{Question}

\section{Introduction}\label{section: introduction}
\subsection{Background} Let $G$ be a finitely generated group with identity element $1\in G$. A nonnegative function $\ell:G\to \N=\{0,1,2,\ldots\}$ is called a proper length function if
\begin{enumerate}
    \item $\ell(gh)\le \ell(g)+\ell(h)$, for all $g,h\in G$;
    \item $\ell(g)=\ell(g^{-1})$, for all $g\in G$;
    \item $\ell(1)=0$;
    \item for every $r\in \N$, the ball $B_r=\{g\in G: \ell(g)\le r\}$ is finite.
\end{enumerate}
We denote the sphere of radius $r$ by $S_r=\{g\in G: \ell(g)=r\}$. Throughout the paper, all length functions are assumed to be proper.
For a finite symmetric generating set $S\subseteq G$, the associated word length is given by
\begin{align*}
    \ell_S(g)=\min\{n\ge 1: g=s_1\cdots s_n, s_1,\ldots,s_n\in S\};\quad \ell_S(1)=0.
\end{align*}
We say that $G$ has the rapid decay property $\RD$ if there exists a polynomial $P$ such that, for every positive integer $r$ and every $f\in \C[G]$ supported on $S_r$,
\begin{align}\label{equation: RD definition}
    \norm{f}_*\le P(r)\norm{f}_2.
\end{align}
Here, $\norm{f}_*$ is the operator norm of left multiplication by $f$ on $\ell^2(G)$, and $\norm{f}_2$ is the usual $\ell^2$-norm. Haagerup's original argument \cite{Haagerup1978originalpaper} proves \eqref{equation: RD definition} for the free group $\mathbb{F}_L$, with the optimal polynomial $P(r)=r+1$. Jolissaint later extended this estimate to several classes of groups, including hyperbolic groups \cite{Jolissaint1990RD}. Dru{\c{t}}u and Sapir established the corresponding permanence result for groups that are hyperbolic relative to subgroups with rapid decay \cite{DrutuSapirrelativehyperbolic}. Since then, the rapid decay property has been proved for many other discrete groups; see the survey \cite{Chatterji2017Introduction}.

The rapid decay property has important applications in geometric group theory and operator algebras. Haagerup used it to establish the metric approximation property for reduced free group $C^*$-algebras \cite{Haagerup1978originalpaper}. Connes and Moscovici used it in their proof of the Novikov conjecture for Gromov hyperbolic groups \cite{ConnesHenri1990Novikovconjecture}, and it also played a role in major results concerning the Baum--Connes conjecture \cite{Lafforgue2002BaumConnes}.

In the operator-valued setting, one asks whether analogous estimates hold uniformly across matrix levels. Haagerup and Pisier \cite{HaagerupPisier1993operatorspace} established such a phenomenon for functions supported on the generators of the free group $\mathbb{F}_L$. Let $B(\Hil)$ denote the algebra of bounded operators on a Hilbert space $\Hil$, let $\lambda:\mathbb{F}_L\to \mathcal{U}(\ell^2(\mathbb{F}_L))$ be the left-regular representation, and let $C_r^*(\mathbb{F}_L)$ be the $C^*$-algebra generated by $\lambda$. We equip $B(\Hil)\otimes C_r^*(\mathbb{F}_L)$ with the minimal tensor norm. If $S$ is the set of generators of $\mathbb{F}_L$ and their inverses, and $(a_g)_{g\in S}\subseteq B(\Hil)$, then the following Khintchine inequality holds:
\begin{align*}
    \max \left\{\norm{\sum_{g\in S}a_ga_g^*},\norm{\sum_{g\in S}a_g^*a_g}\right\}^{\frac{1}{2}}
    &\le \norm{\sum_{g\in S}a_g\otimes \lambda(g)}\\
    &\le 2\max \left\{\norm{\sum_{g\in S}a_ga_g^*},\norm{\sum_{g\in S}a_g^*a_g}\right\}^{\frac{1}{2}}.
\end{align*}
Buchholz later extended this result to an operator-valued Haagerup inequality for homogeneous functions on free groups \cite{Buchholz1999OVfree}. To formulate the estimate, let $G$ be a group, let $\Hil$ be a Hilbert space, and let $f:G\to B(\Hil)$ have finite support:
\begin{align*}
    \operatorname{supp}(f)=\{g\in G: f(g)\ne 0\}.
\end{align*}
We denote
\begin{align}\label{equation: lambda(f)}
    \lambda(f)=\sum_{g\in G}f(g)\otimes \lambda(g)\in B(\Hil\otimes \ell^2(G)).
\end{align}
For $i,j\in \N$, define the spherical operators
\begin{align*}
    B_{i,j}(f)=(f(gh^{-1}))_{g\in S_i, h\in S_j}: \Hil^{S_j} \to \Hil^{S_i}.
\end{align*}
Buchholz proved that, for every $f:G\to B(\Hil)$ supported in $S_r$,
\begin{align*}
    \max_{0\le k\le r}\norm{B_{k,r-k}(f)}\le \norm{\lambda(f)}\le \sum_{k=0}^r \norm{B_{k,r-k}(f)}.
\end{align*}
Beyond the free-group case, Ricard and Xu established fixed-length Khintchine inequalities for reduced free products \cite{RicardXu2006reducedfreeproducts}. Caspers, Klisse, and Larsen developed graph-product analogues for $C^*$-algebras and Hecke algebras \cite{CaspersKlisseLarsen2021graphproduct}, while Ciobanu, Holt, and Rees proved that scalar rapid decay is preserved under graph products \cite{CiobanyHoltRees2013rdgraphproducts}. These results are especially relevant to the right-angled Artin group application below: graph-product methods exploit algebraic normal forms, whereas our estimates arise from centroid geometry and are expressed in terms of the spherical blocks $B_{i,j}(f)$. Toyota and Yang extended the operator-valued Haagerup inequality to general hyperbolic spaces \cite{ToyotaYang2026OVhyperbolic}. The operator-space formulation makes the polynomial factor transparent: in the free case, Buchholz's estimate is governed by the $r+1$ block operators $B_{k,r-k}(f)$, $0\le k\le r$. In a different direction, strong Haagerup inequalities have been obtained for holomorphic words and, subsequently, with operator coefficients \cite{KempSpeicher2007strongHaagerupscalar,delaSalle2009strongHaagerup}.

\subsection{Main results}
The paper's central contribution is a centroid-based mechanism that converts geometric counting estimates into operator-valued norm bounds. It extends the operator-valued Haagerup inequalities of Buchholz for free groups and Toyota and Yang for hyperbolic groups to median and coarse median geometry. Examples include limit groups, mapping class groups of oriented surfaces, groups acting "nicely" on finite-rank $\CAT$ cube complexes, and relatively hyperbolic groups with respect to those previous groups.
Because the relevant definitions are technical, we defer them to Sections~\ref{section: median spaces} and \ref{section: coarse median spaces} and state the main results first.
Our first class of examples consists of median groups. A group $G$ with finite generating set $S$ is called a median group if its Cayley graph $(G, S)$, equipped with the word metric, is a median metric space.
\begin{theorem}\label{theorem: median groups}
    Let $G$ be a median group of rank $\nu$. Then, for every $f:G\to B(\Hil)$ supported in $S_r$,
    \begin{align*}
        \max_{t\in \{0,\ldots, r\}}\norm{B_{t,r-t}(f)}\le \norm{\lambda(f)}\le \binom{\nu+r}{r}\max_{t\in \{0,\ldots, r\}}\norm{B_{t,r-t}(f)}.
    \end{align*}
\end{theorem}
Right-angled Artin groups provide a concrete family of median groups. If $\Gamma=(V,E)$ is a finite graph, its right-angled Artin group is
\begin{align*}
    G(\Gamma)=\langle g_v: v\in V\mid [g_u,g_v]=1\text{ for every }(u,v)\in E\rangle.
\end{align*}
Theorem~\ref{theorem: median groups} gives the polynomial factor $\binom{\omega(\Gamma)+r}{r}$, where $\omega(\Gamma)$ is the clique number of $\Gamma$; see Corollary~\ref{corollary: right angled artin groups}.

Theorem~\ref{theorem: median groups} does not follow from the scalar rapid decay or existing graph-product Khintchine inequalities \cite{CaspersKlisseLarsen2021graphproduct}. Those graph-product results exploit algebraic normal forms, whereas our argument uses only median geometry and therefore applies to arbitrary finite-rank median groups and to groups acting freely on finite-rank $\CAT$ cube complexes. However, two obstructions remain: median spaces need not have unique geodesics, and scalar rapid-decay estimates do not pass directly to operator-valued coefficients. The main operator-algebraic ingredient is Proposition~\ref{proposition: main decomposition-Haagerup}, a reusable block decomposition of $\lambda(f)$. In the median setting, it shows that the same spherical slices of intervals that govern scalar counting estimates also govern the operator-valued upper bound. A rectangular factorization of the associated Schur multipliers then gives the complementary lower bound; see Lemma~\ref{lemma: lower bound median}.

We next extend the result to coarse median spaces; see \cite{Bowditch2013coarsemedian}. Informally, a metric space $(X,d)$ is coarse median if it admits a coarsely Lipschitz ternary operation and every finite subset has a finite median approximation with constants depending only on its cardinality. Further developments in coarse median geometry, including the structure of coarse intervals, rank, and iterated medians, can be found in \cite{SpakulaWright2017propertyA, NibloWrightZang2019fourpoint, NibloWrightZang2021}. A finitely generated group $G$ is coarse median if some (equivalently, every) Cayley graph associated with a finite generating set is a coarse median space. This notion extends hyperbolicity to higher-rank settings.
\begin{theorem}\label{theorem: coarse median groups}
    Let $G$ be a coarse median group of rank at most $\nu$. Then there exist a function $Q:\N\to\R_+$ satisfying
    \begin{align*}
        Q(r)=o(r^{1+\nu/2+\varepsilon})\qquad\text{for every }\varepsilon>0,
    \end{align*}
    and an affine function $L:\N\to\R_+$ such that, for every $f:G\to B(\Hil)$ supported in $S_r$,
    \begin{align*}
        \norm{\lambda(f)}\le Q(r)\max_{\substack{0\le i \le L(r)\\ |i-j|\le r\le i+j}}\norm{B_{i,j}(f)}.
    \end{align*}
\end{theorem}
The larger range of pairs $(i,j)$, compared with the relation $r=i+j$ in the median case, is offset by the exponent $1+\nu/2\le\nu$ for $\nu\ge2$. Together with \cite{ToyotaYang2026OVhyperbolic}, this shows that every finite-rank coarse median group satisfies an operator-valued Haagerup inequality. Examples of coarse median groups include:
\begin{enumerate}
    \item Mapping class groups of oriented surfaces \cite[Theorem 8.2]{Bowditch2014embedding}.
    \item Groups that are relatively hyperbolic with respect to coarse median groups \cite[Theorem 1.1]{Bowditch2013invariancehyperbolicity}.
    \item Direct products of coarse median groups \cite[Example 1.4.2.3]{Bowditch2019notes}.
    \item Limit groups \cite[Example 1.4.2.5]{Bowditch2019notes}.
    \item Groups admitting a proper cellular action, with uniformly bounded stabilizers, on a finite-rank $\CAT$ cube complex \cite[Corollary 3.5, Remark 3.6]{Sapir2015centroids}.
\end{enumerate}
In particular, Theorem \ref{theorem: coarse median groups} applies to the groups listed above.

The coarse median theorem requires a genuinely new geometric input. A natural attempt would be to approximate a coarse median space by a finite median model and then apply Theorem~\ref{theorem: median groups}. This procedure is effective in the scalar setting, where the estimates reduce to the cardinality of coarse intervals; see \cite[Section 9]{Bowditch2014embedding}. In contrast, it is insufficient here: noncommutativity prevents a direct passage to spherical blocks, while the non-uniqueness of geodesics and medians makes the relevant block range unstable. Even in the hyperbolic case, the relation $i+j=r$ must be thickened to $r\le i+j\le r+\delta+1$; in higher-rank coarse median spaces, no uniformly bounded thickening suffices.

Our solution is a cardinality-independent coarse iterate median construction. For $Z=\{x_1,\ldots,x_n\}\subseteq X$, the iterated median is the unique point $q\in X$ satisfying
\begin{align*}
    \bigcap_{i=1}^n [x_i,x_0]=[q,x_0].
\end{align*}
In rank $\nu$, the point $q$ is already determined by at most $\nu$ of the inputs. A direct coarse analogue would normally accumulate an error proportional to the number of points. Proposition~\ref{proposition: coarse meet existence} avoids this accumulation: for every bounded set $Z\subseteq X$, it produces a single point lying in a uniformly thickened interval $[x_0,z]$ for every $z\in Z$, with an error independent of $|Z|$ and with displacement controlled linearly by the radius of $Z$. The proof uses an iterative error-reduction argument that decreases the interval error geometrically while keeping the total displacement linear. This construction is the key geometric ingredient behind Theorem~\ref{theorem: coarse median groups}; compare Lemma~\ref{lemma: distance of meets} in the exact median setting.

\subsection{Centroid Haagerup inequalities} At the general level, we identify the operator-valued content of the centroid method. Centroid constructions already appear in Behrstock and Minsky's proof of rapid decay for mapping class groups \cite{BehrstockMinsky2011centroid}, and Sapir subsequently isolated an abstract centroid property encompassing this strategy \cite{Sapir2015centroids}; the method is also related in spirit to Chatterji and Ruane's median-type approach \cite{ChatterjiRuane2005median}. Our parametrized decomposition turns these scalar counting conditions into block operators $M_t(f)$ and makes each centroid condition control a distinct operator-norm factor. Thus, the framework does not merely prove a polynomial estimate: it records where each part of the polynomial loss comes from.

Let $(X,d)$ be a metric space with base point $x_0\in X$. Throughout the paper, we assume $(X,d)$ is a connected, uniformly locally finite graph endowed with its graph metric. Let $G$ be a finitely generated group acting isometrically on $X$, and define the orbit length
\begin{align*}
    \ell(g)=d(x_0,g\cdot x_0),\qquad g\in G.
\end{align*}
We assume that this length is proper. Because $X$ is uniformly locally finite, the orbit length is proper if and only if the stabilizer $\operatorname{Stab}_G(x_0)$ is finite.

Let $\mathfrak{c}:G\times G\to X$ be a map. We refer to $X$ as the centroid space and set
\begin{align*}
    \mathcal{C}_p=\{\mathfrak{c}(p,h):h\in G\},\quad p\in G.
\end{align*}
We say that $(G,\mathfrak{c})$ satisfies the centroid property (C), with polynomial bounds $P_1, P_2, P_3:\N\to\N$, if
\begin{enumerate}[label=(C\arabic*)]
    \item\label{property c1} For every $h\in G$ and $r\in \N$, 
    \begin{align*}
        \#\{\mathfrak{c}(g,h): \ell(g)=r\} \le P_1(r).
    \end{align*}
    \item\label{property c2} For every $g\in G$, $|\mathcal{C}_g|\le P_2(\ell(g))$.
    \item\label{property c3} For every $h\in G$ and $r\in \N$,
    \begin{align*}
        \#\{g^{-1}\cdot\mathfrak{c}(g,gh): \ell(g)=r\}\le P_3(r).
    \end{align*}
\end{enumerate}
We formulate Sapir's centroid conditions using spheres rather than balls; the two formulations are polynomially equivalent. One of the main ideas of this paper is to stratify the centroid fibers by auxiliary parameters and decompose the left-regular representation accordingly.

Let $(\Theta_r)_{r\in\N}$ be a family of nonempty finite sets. For each $g\in G$, suppose that a map $\theta_g:\mathcal{C}_g\to\Theta_{\ell(g)}$ is given. We call $(\theta_g)_{g\in G}$ the parametrization maps, $(\Theta_r)_{r\in\N}$ the parametrization spaces, and $(\theta,\Theta)=(\theta_g,\Theta_r)_{g\in G,r\in\N}$ a parametrization family for $(G,\mathfrak{c})$.

The parametrization maps are not part of Sapir's definition and, abstractly, impose no additional restriction. Indeed, Property~\ref{property c2} provides sets $(\mathcal{X}_r)_{r\in\N}$ with $|\mathcal{X}_r|=P_2(r)$ and injective maps $\theta_g:\mathcal{C}_g\to\mathcal{X}_{\ell(g)}$. When chosen geometrically, however, the parametrizations encode useful information about the group, as we will see in Section~\ref{section: median spaces}.

Sapir proved that groups $(G,\mathfrak{c})$ satisfying (C) also have property (RD), with polynomial growth rate $(P_1P_2P_3)^{1/2}$ \cite[Theorem 2.3]{Sapir2015centroids}.

The following proposition is the paper's main blueprint for operator-valued Haagerup inequalities. It converts a parametrized centroid map into a block-diagonal decomposition of $\lambda(f)$ and separates the geometric counting terms from the norms of the individual blocks.
\begin{proposition}\label{proposition: main decomposition-Haagerup}
    Let $G$ be a finitely generated group acting isometrically on $(X,d)$ and $r\in \N$. Let $\mathfrak{c}: G\times G\to X$ be a centroid map.
    Let $f:G\to B(\Hil)$ be supported in $S_r$, and let $(\theta,\Theta)$ be a parametrization family for $(G,\mathfrak{c})$.
    Then, there are operators $U_t,V_t:\ell^2(G)\to\ell^2(X\times G)$, independent of $f$, and a family of operators  
    \begin{align}\label{equation: M_x}
        M_t(f):=\bigoplus_{c\in X}M_{c,t}(f) \in B(\Hil\otimes \ell^2(X\times G));\quad t\in \Theta_r,
    \end{align}
    such that the following hold.
    \begin{enumerate}[label=(\thetheorem.\arabic*)]
        \item\label{item: norm of U} For every $t\in \Theta_r$, 
        \begin{align*}
            \norm{U_t}^2=\sup_{h\in G}\#\{\mathfrak{c}(g,h)|\ g\in S_r,\ \theta_g(\mathfrak{c}(g,h))=t\}<\infty.
        \end{align*}
        \item\label{item: norm of V} For every $t\in \Theta_r$,
        \begin{align*}
            \norm{V_t}^2=\sup_{h\in G}\#\{g^{-1}\cdot \mathfrak{c}(g,gh)|\ g\in S_r,\ \theta_g(\mathfrak{c}(g,gh))=t\}<\infty.
        \end{align*}
        \item\label{item: decomposition} We have 
        \begin{align*}
            \lambda(f)=\sum_{t\in \Theta_r}(\Id_{\Hil}\otimes U_t)^* M_{t}(f)(\Id_{\Hil}\otimes V_t).
        \end{align*}
        In particular,
        \begin{align*}
            \norm{\lambda(f)}\le \sum_{t\in \Theta_r}\norm{U_t}\norm{V_t}\sup_{c\in X}\norm{M_{c,t}(f)}.
        \end{align*}
    \end{enumerate}
\end{proposition}
Under the centroid conditions, we have the following Haagerup-type inequality.
\begin{theorem}\label{theorem: main Haagerup inequality}
    Let $G$ be a finitely generated group acting isometrically on $(X,d)$. Let $\mathfrak{c}: G\times G\to X$ be a centroid map.
    Let $f:G\to B(\Hil)$ be supported in $S_r$. Let $(\theta,\Theta)$ be a parametrization family for $(G,\mathfrak{c})$, and let $(U_t,V_t,M_t)_{t\in \Theta_r}$ be the operators constructed in Proposition~\ref{proposition: main decomposition-Haagerup}. If Properties~\ref{property c1} and \ref{property c3} hold, then
    \begin{align*}
        \norm{\lambda(f)}\le \sqrt{P_1(r)P_3(r)}\sum_{t\in \Theta_r}\sup_{c\in X}\norm{M_{c,t}(f)}.
    \end{align*}
    Moreover, if $|\Theta_r|\le P_2(r)$ for every $r\in\N$, then
    \begin{align*}
        \norm{\lambda(f)}\le P_2(r)\sqrt{P_1(r)P_3(r)}\sup_{(c,t)\in X\times \Theta_r}\norm{M_{c,t}(f)}.
    \end{align*}
\end{theorem}
The theorem makes the role of the centroid hypotheses explicit: Property~\ref{property c1} controls the row embedding $U_t$, Property~\ref{property c3} controls the column embedding $V_t$, and, for surjective parametrizations, Property~\ref{property c2} controls the number of layers $|\Theta_r|$. This separation allows the scalar centroid method to extend to the operator-valued setting.

\subsection*{Outline of the paper} Section~\ref{section: proofs} proves the parametrized centroid decompositions in Proposition \ref{proposition: main decomposition-Haagerup}, and then derives the general operator-valued Haagerup estimate in Theorem~\ref{theorem: main Haagerup inequality}. Section~\ref{section: lower bounds} establishes complementary lower bounds by expressing the block operators $M_{c,t}(f)$ as Schur multipliers and applying Grothendieck-type factorization. Section~\ref{section: median spaces} specializes the framework to median metric spaces, relates the centroid blocks to Buchholz's spherical operators, proves Theorem~\ref{theorem: median groups}, and gives applications to groups acting on finite-rank $\CAT$ cube complexes, including right-angled Artin groups. Finally, Section~\ref{section: coarse median spaces} develops the coarse iterated-median construction, verifies the corresponding coarse iterate median and centroid estimates, and proves the operator-valued Haagerup inequality stated in Theorem~\ref{theorem: coarse median groups}.

\section{Operator decompositions}\label{section: proofs}
For $(c,g)\in X\times G$, define $V_{c,g}\in B(\ell^2(G))$ on the canonical basis by
\begin{align*}
    V_{c,g}\delta_h=\begin{cases}
        \delta_{gh},&\text{ if }\mathfrak{c}(g,gh)=c,\\
        0,&\text{ otherwise}.
    \end{cases}
\end{align*}
Thus, $V_{c,g}$ agrees with the shift $\lambda(g)$ on the basis vectors for which the centroid condition is satisfied. 

We begin with a parametrized decomposition of the left-regular representation $\lambda$ of $G$, analogous to \cite[Lemma 2.4]{Buchholz1999OVfree}.
\begin{proposition}\label{proposition: decomposition}
    Let $G$ be a finitely generated group acting isometrically on a metric space $(X,d)$. Let $\mathfrak{c}:G^2\to X$ be a centroid map and $(\theta,\Theta)$ be a parametrization family for $(G,\mathfrak{c})$.
    Then, for any $g\in G$, we have
    \begin{align*}
        \lambda(g)=\sum_{c\in \mathcal{C}_g}V_{c,g}=\sum_{t\in \Theta_{\ell(g)}}\sum_{c\in \mathcal{C}_g}\mathbf{1}_{\theta_g(c)=t} V_{c,g},
    \end{align*}
    where both equalities are understood in the strong operator topology.
    Moreover, if Property~\ref{property c2} holds, the first sum contains at most $P_2(\ell(g))$ terms.
\end{proposition}
\begin{proof}
    Fix $h\in G$. Then $\lambda(g)\delta_h=\delta_{gh}$, whereas
    \begin{align*}
        \left(\sum_{c\in \mathcal{C}_g}V_{c,g}\right)\delta_h=\sum_{c\in \mathcal{C}_g}\mathbf{1}_{c=\mathfrak{c}(g,gh)}\delta_{gh}.
    \end{align*}
    Since $\mathfrak{c}$ is single-valued and $\mathfrak{c}(g,gh)\in\mathcal{C}_g$, exactly one summand is nonzero. Hence
    \begin{align*}
        \left(\sum_{c\in \mathcal{C}_g}V_{c,g}\right)\delta_h=\delta_{gh}.
    \end{align*}
    The second equality follows by partitioning $\mathcal{C}_g$ according to the values of $\theta_g$:
    \begin{align*}
        \sum_{c\in \mathcal{C}_g}V_{c,g}=\sum_{t\in \Theta_{\ell(g)}}\sum_{\substack{c\in \mathcal{C}_g\\ \theta_g(c)=t}} V_{c,g}.
    \end{align*}
    Moreover, the summands have pairwise disjoint initial supports on the canonical basis, so the partial sums converge strongly to $\lambda(g)$.
\end{proof}
For $f:G\to B(\Hil)$ supported in $S_r$, let $(\lambda(f))_{g,h}\in B(\Hil)$ denote the matrix entries determined by
\begin{align*}
    \lambda(f) (a \otimes \delta_h)=\sum_{g\in G} (\lambda(f))_{g,h}a\otimes \delta_g;\quad a\in \Hil.
\end{align*}
A direct computation gives
\begin{align*}
     (\lambda(f))_{g,h}=f(gh^{-1}),
\end{align*}
for all $g,h\in G$. 
By Proposition~\ref{proposition: decomposition}, we have
\begin{align*}
    \lambda(f)=\sum_{t\in \Theta_r}\sum_{g\in G}\sum_{\substack{c\in \mathcal{C}_g\\ \theta_g(c)=t}} f(g)\otimes V_{c,g}.
\end{align*}
For convenience, set
\begin{align*}
    T_t=\sum_{p\in G}\sum_{\substack{c\in \mathcal{C}_p\\ \theta_p(c)=t}} f(p)\otimes V_{c,p},
\end{align*}
for $t\in\Theta_r$. Its matrix entries are obtained from
\begin{align*}
    T_t(a \otimes \delta_h)=\sum_{p\in G}\sum_{c\in \mathcal{C}_p}\mathbf{1}_{\theta_p(c)=t}\mathbf{1}_{c=\mathfrak{c}(p,ph)} f(p)a\otimes \delta_{ph},
\end{align*}
for $a\in \Hil$ and $h\in G$.
The two indicator functions combine to give 
\begin{align*}
    \sum_{c\in \mathcal{C}_p}\mathbf{1}_{\theta_p(c)=t}\mathbf{1}_{c=\mathfrak{c}(p,ph)}=\mathbf{1}_{\theta_p(\mathfrak{c}(p,ph))=t}.
\end{align*}
After substituting $p=gh^{-1}$, we obtain
\begin{align*}
    (T_t)_{g,h}=\mathbf{1}_{\theta_{gh^{-1}}(\mathfrak{c}(gh^{-1},g))=t} f(gh^{-1}).
\end{align*}
We now define the operators $M_{c,t}$ from Proposition~\ref{proposition: main decomposition-Haagerup}. For each $c\in X$, the operator $M_{c,t}$ selects the layer of $T_t$ corresponding to the normalized centroid $g^{-1}\cdot\mathfrak{c}(gh^{-1},g)$. Explicitly, 
\begin{align}\label{equation: operators M_c,x}
    (M_{c,t}(f))_{g,h}:=\mathbf{1}_{c=g^{-1}\cdot\mathfrak{c}(gh^{-1},g)} (T_t)_{g,h}=\mathbf{1}_{c=g^{-1}\cdot\mathfrak{c}(gh^{-1},g)}\mathbf{1}_{\theta_{gh^{-1}}(\mathfrak{c}(gh^{-1},g))=t} f(gh^{-1}),
\end{align}
for every $(c,t)\in X\times \Theta_r$ and $g,h\in G$. The operator $M_t(f)$ from \eqref{equation: M_x} acts by
\begin{align}\label{equation: action of M_x}
    M_t(f) (a\otimes \delta_{(c,h)})=\sum_{g\in G} (M_{c,t}(f))_{g,h}a \otimes \delta_{(c,g)},
\end{align}
for $a\in\Hil$ and $(c,h)\in X\times G$. Finally, define $U_t,V_t:\ell^2(G)\to\ell^2(X\times G)$ by
\begin{equation}
\label{equation: UV}
\begin{aligned}
    U_t\delta_g
    &= \sum_{c\in A_t(g)}\delta_{(c,g)};
    &\qquad
    A_t(g)
    &= \{g^{-1}\cdot \mathfrak{c}(p,g): \ell(p)=r; \theta_p(\mathfrak{c}(p,g))=t\};\\
    V_t\delta_g
    &= \sum_{c\in B_t(g)}\delta_{(c,g)};
    &\qquad
    B_t(g)
    &= \{g^{-1}p^{-1}\cdot \mathfrak{c}(p,pg): \ell(p)=r; \theta_p(\mathfrak{c}(p,pg))=t\}.
\end{aligned}
\end{equation}
Consequently,
\begin{align}\label{equation: adjoint of U}
    U^*_t\delta_{(c,g)}=\mathbf{1}_{c\in A_t(g)}\delta_g,
\end{align}
for every $(c,g)\in X\times G$. We prove Proposition~\ref{proposition: main decomposition-Haagerup} and Theorem~\ref{theorem: main Haagerup inequality} using the following lemmas.
\begin{lemma}\label{lemma: bound on norm of U_t and V_t}
    Let $t\in\Theta_r$ and let $U_t,V_t$ be as in \eqref{equation: UV}. Then Items~\ref{item: norm of U} and \ref{item: norm of V} hold. Moreover,
    \begin{align*}
        &\norm{U_t}^2=\sup_{h\in G}|A_t(h)|\le\sup_{h\in G}\#\{\mathfrak{c}(g,h): \ell(g)=r\} \\
        &\norm{V_t}^2=\sup_{h\in G}|B_t(h)|\le\sup_{h\in G}\#\{g^{-1}\cdot\mathfrak{c}(g,gh): \ell(g)=r\}.
    \end{align*}
    In particular, if Properties~\ref{property c1} and \ref{property c3} hold, we have
    \begin{align*}
        &\norm{U_t}^2\le P_1(r)\\
        &\norm{V_t}^2\le P_3(r).
    \end{align*}
\end{lemma}
\begin{proof}
By \eqref{equation: UV}, we have
\begin{align*}
    &\norm{U_t}^2=\sup_{h\in G}|A_t(h)|\\
    &\norm{V_t}^2=\sup_{h\in G}|B_t(h)|.
\end{align*}
Each element of $S_r$ contributes at most one point to either set, so both suprema are bounded by $|S_r|$ and are therefore finite. Fix $h\in G$. Since the action of $h^{-1}$ on $X$ is injective,
\begin{align*}
    |A_t(h)|=\#\{\mathfrak{c}(p,h):\ell(p)=r,\ \theta_p(\mathfrak{c}(p,h))=t\},
\end{align*}
which proves Item~\ref{item: norm of U}. Likewise, $B_t(h)$ is the image under $h^{-1}$ of the set of points $p^{-1}\cdot\mathfrak{c}(p,ph)$ satisfying the stated conditions. Hence
\begin{align*}
    |B_t(h)|=\#\{p^{-1}\cdot\mathfrak{c}(p,ph):\ell(p)=r,\ \theta_p(\mathfrak{c}(p,ph))=t\},
\end{align*}
which proves Item~\ref{item: norm of V}. Dropping the conditions involving $\theta_p$ gives
\begin{align*}
    &\sup_{h\in G}|A_t(h)|\le \sup_{h\in G}\#\{\mathfrak{c}(p,h): \ell(p)=r\}\\
    &\sup_{h\in G}|B_t(h)|\le \sup_{h\in G}\#\{p^{-1}\cdot \mathfrak{c}(p,ph):\ell(p)=r\}.
\end{align*}
The polynomial bounds for $\norm{U_t}^2$ and $\norm{V_t}^2$ now follow from Properties~\ref{property c1} and \ref{property c3}.
\end{proof}
The next lemma proves Item~\ref{item: decomposition}.
\begin{lemma}\label{lemma: decomposition of lambda f}
    Let $U_t,V_t,M_t$ be as in \eqref{equation: action of M_x} and \eqref{equation: UV}. Then,
    \begin{align}\label{equation: Item decomposition in the lemma}
        \lambda(f)=\sum_{t\in \Theta_r} (\Id_{\Hil}\otimes U_t)^* M_{t}(f)(\Id_{\Hil}\otimes V_t).
    \end{align}
\end{lemma}
\begin{proof}
Fix $a\in\Hil$ and $h\in G$. Then
\begin{align*}
    \lambda(f)a\otimes \delta_h=\sum_{g\in G}f(gh^{-1})a\otimes \delta_g.
\end{align*}
We compare this expression with the right-hand side of \eqref{equation: Item decomposition in the lemma}. For fixed $t\in\Theta_r$, \eqref{equation: UV} gives
\begin{align*}
    (\Id_{\Hil}\otimes V_t)a\otimes \delta_h= \sum_{c\in B_t(h)} a\otimes \delta_{(c,h)}.
\end{align*}
Applying \eqref{equation: action of M_x} yields
\begin{align*}
    M_{t}(f)\sum_{c\in B_t(h)} a\otimes \delta_{(c,h)}=\sum_{c\in B_t(h)}\sum_{g\in G}(M_{c,t}(f))_{g,h} a\otimes \delta_{(c,g)}.
\end{align*}
Finally, \eqref{equation: adjoint of U} gives
\begin{align*}
    \sum_{t\in \Theta_r}\left((\Id_{\Hil}\otimes U_t)^* M_{t}(f)(\Id_{\Hil}\otimes V_t)\right)a\otimes \delta_h=\sum_{t\in \Theta_r}\sum_{g\in G}\sum_{c\in B_t(h)\cap A_t(g)}(M_{c,t}(f))_{g,h} a\otimes \delta_g.
\end{align*}
It therefore suffices to prove that
\begin{align*}
    f(gh^{-1})=\sum_{t\in \Theta_r}\sum_{c\in B_t(h)\cap A_t(g)}(M_{c,t}(f))_{g,h},
\end{align*}
for all $g,h\in G$. If $f(gh^{-1})\ne0$, then $gh^{-1}\in S_r$. Thus it remains to verify that, whenever $gh^{-1}\in S_r$,
\begin{align*}
    1=\sum_{t\in \Theta_r}\sum_{c\in B_t(h)\cap A_t(g)}\mathbf{1}_{c=g^{-1}\cdot\mathfrak{c}(gh^{-1},g)}\mathbf{1}_{\theta_{gh^{-1}}(\mathfrak{c}(gh^{-1},g))=t}.
\end{align*}
Fix $g,h\in G$ such that $gh^{-1}\in S_r$. By definition, $t^*:=\theta_{gh^{-1}}(\mathfrak{c}(gh^{-1},g))\in \Theta_r$.
It is therefore enough to show that
\begin{align*}
    1=\sum_{c\in B_{t^*}(h)\cap A_{t^*}(g)}\mathbf{1}_{c=g^{-1}\cdot\mathfrak{c}(gh^{-1},g)}.
\end{align*}
Equivalently, we must show that $g^{-1}\cdot\mathfrak{c}(gh^{-1},g)\in B_{t^*}(h)\cap A_{t^*}(g)$. Indeed, let $p=gh^{-1}$. First,
\begin{align*}
    &g^{-1}\cdot\mathfrak{c}(gh^{-1},g)=g^{-1}\cdot \mathfrak{c}(p,g);\quad \theta_p(\mathfrak{c}(p,g))=t^*;\quad \ell(p)=r.
\end{align*}
Thus, $g^{-1}\cdot\mathfrak{c}(gh^{-1},g)\in A_{t^*}(g)$. Likewise,
\begin{align*}
    g^{-1}\cdot\mathfrak{c}(gh^{-1},g)=h^{-1}p^{-1}\cdot \mathfrak{c}(p,ph);\quad \theta_p(\mathfrak{c}(p,ph))=t^*;\quad \ell(p)=r.
\end{align*}
Hence $g^{-1}\cdot\mathfrak{c}(gh^{-1},g)\in B_{t^*}(h)$ as well.
\end{proof}
We can now prove Proposition~\ref{proposition: main decomposition-Haagerup} and Theorem~\ref{theorem: main Haagerup inequality}.
\begin{proof}[Proof of Proposition~\ref{proposition: main decomposition-Haagerup}]
    Lemma~\ref{lemma: bound on norm of U_t and V_t} proves Items~\ref{item: norm of U} and \ref{item: norm of V}, while Lemma~\ref{lemma: decomposition of lambda f} proves Item~\ref{item: decomposition}. Taking norms and applying the triangle inequality gives
    \begin{align*}
        \norm{\lambda(f)}\le \sum_{t\in \Theta_r} \norm{\Id_\Hil\otimes U_t}\norm{\Id_\Hil\otimes V_t}\norm{M_t(f)}.
    \end{align*}
    Since $M_t(f)$ is the orthogonal direct sum of the blocks $M_{c,t}(f)$, we have $\norm{M_t(f)}=\sup_{c\in X}\norm{M_{c,t}(f)}$.
\end{proof}
\begin{proof}[Proof of Theorem~\ref{theorem: main Haagerup inequality}]
    Proposition~\ref{proposition: main decomposition-Haagerup} gives
    \begin{align*}
         \norm{\lambda(f)}\le \sum_{t\in \Theta_r} \norm{U_t}\norm{V_t}\sup_{c\in X}\norm{M_{c,t}(f)}.
    \end{align*}
    Under Properties~\ref{property c1} and \ref{property c3}, Lemma~\ref{lemma: bound on norm of U_t and V_t} shows that
    \begin{align*}
        \norm{\lambda(f)}\le \sqrt{P_1(r)P_3(r)}\sum_{t\in \Theta_r}\sup_{c\in X}\norm{M_{c,t}(f)}.
    \end{align*}
    The second estimate follows from the $\ell^1$--$\ell^\infty$ bound and the assumption $|\Theta_r|\le P_2(r)$.
\end{proof}

\section{Lower bounds}\label{section: lower bounds}
Throughout this section, $f: G\to B(\Hil)$ is assumed to be supported in $S_r$. We derive several lower bounds for
\begin{align*}
    \lambda(f)=\sum_{g\in G}f(g)\otimes \lambda(g).
\end{align*}

We impose an additional factorization condition linking the centroid map $\mathfrak{c}$ to the parametrization maps $(\theta_g)_{g\in G}$. The following elementary observation is the starting point. 
\begin{lemma}\label{lemma: Schur decomposition}
    Let $(c,t)\in X\times \Theta_r$.
    Let $M_{c,t}(f)$ be defined by \eqref{equation: operators M_c,x}, and let $S_{c,t}=((S_{c,t})_{g,h})_{g,h\in G}$ be the $\{0,1\}$-valued matrix
    \begin{align*}
        (S_{c,t})_{g,h}:=\mathbf{1}_{c=g^{-1}\cdot\mathfrak{c}(gh^{-1},g)}\mathbf{1}_{\theta_{gh^{-1}}(\mathfrak{c}(gh^{-1},g))=t}\mathbf{1}_{gh^{-1}\in S_r}.
    \end{align*}
    Then,
    \begin{align*}
        (M_{c,t}(f))_{g,h}=(S_{c,t})_{g,h} \lambda(f)_{g,h},
    \end{align*}
    for all $g,h\in G$. Thus $M_{c,t}$ is obtained from $\lambda(f)$ by Schur multiplication with $S_{c,t}$; we write $M_{c,t}=S_{c,t}\circ\lambda(f)$.
\end{lemma}
We may therefore use the theory of Schur multipliers and Grothendieck factorization; see \cite[Chapter 5]{Pisier2001}. Let $\mathcal{K}$ be a Hilbert space with orthonormal basis $(e_p)_{p\in\N}$, and let $S=(S_{p,q})_{p,q\in\N}$ be a scalar matrix. The Schur multiplier $T\mapsto S\circ T$ is bounded on $B(\mathcal{K})$ whenever there exist an auxiliary Hilbert space $\mathcal{S}$ and families $(x_p)_{p\in\N},(y_q)_{q\in\N}\subseteq\mathcal{S}$ such that
\begin{enumerate}
    \item For all $p,q\in \N$, we have
    \begin{align*}
        S_{p,q}=\langle x_p,y_q\rangle.
    \end{align*}
    \item The families $(x_p)_{p\in\N}$ and $(y_q)_{q\in\N}$ are uniformly bounded:
    \begin{align*}
        \sup_{p\in \N}\norm{x_p},\sup_{q\in \N}\norm{y_q}<\infty.
    \end{align*}
\end{enumerate}
In this case, define
    \begin{align*}
        \norm{S}_{\operatorname{cb}}:=\inf\left\{\left(\sup_{p\in \N}\norm{x_p}\right)\left(\sup_{q\in \N}\norm{y_q}\right): S_{p,q}=\langle x_p,y_q\rangle\right\},
    \end{align*}
where the infimum is taken over all such representations $(\mathcal{S},(x_p)_p,(y_q)_q)$. By \cite[Theorem 5.1]{Pisier2001},
\begin{align*}
    \norm{S\circ T}\le \norm{S}_{\operatorname{cb}}\norm{T}.
\end{align*}
Applying this criterion to the matrices $S_{c,t}$ gives the following consequence.
\begin{corollary}
    Let $(c,t)\in X\times \Theta_r$.
    Let $M_{c,t}(f)$ be defined by \eqref{equation: operators M_c,x}, and let $S_{c,t}=((S_{c,t})_{g,h})_{g,h\in G}$ be the $\{0,1\}$-valued matrix
    \begin{align*}
        (S_{c,t})_{g,h}:=\mathbf{1}_{c=g^{-1}\cdot\mathfrak{c}(gh^{-1},g)}\mathbf{1}_{\theta_{gh^{-1}}(\mathfrak{c}(gh^{-1},g))=t}\mathbf{1}_{gh^{-1}\in S_r}.
    \end{align*}
    If $S_{c,t}$ is a nonzero bounded Schur multiplier, then
    \begin{align*}
        \norm{\lambda(f)} \ge  \frac{1}{\norm{S_{c,t}}_{\operatorname{cb}}}\norm{M_{c,t}(f)}.
    \end{align*}
\end{corollary}
The following more concrete criterion is often sufficient.
\begin{corollary}\label{corollary: lower bound simple version}
    Let $(c,t)\in X\times \Theta_r$.
    Let $M_{c,t}(f)$ be defined by \eqref{equation: operators M_c,x}, and let $S_{c,t}=((S_{c,t})_{g,h})_{g,h\in G}$ be the $\{0,1\}$-valued matrix
    \begin{align*}
        (S_{c,t})_{g,h}:=\mathbf{1}_{c=g^{-1}\cdot\mathfrak{c}(gh^{-1},g)}\mathbf{1}_{\theta_{gh^{-1}}(\mathfrak{c}(gh^{-1},g))=t}\mathbf{1}_{gh^{-1}\in S_r}.
    \end{align*}
    Suppose there is an index set $\mathcal{P}_{c,t}$ and, for each $s\in\mathcal{P}_{c,t}$, subsets $A_{s,c,t},B_{s,c,t}\subseteq G$ such that 
    \begin{align*}
        (S_{c,t})_{g,h}=\sum_{s\in \mathcal{P}_{c,t}} \mathbf{1}_{g\in A_{s,c,t}}\mathbf{1}_{h\in B_{s,c,t}},\qquad \text{whenever }gh^{-1}\in S_r.
    \end{align*}
    Define
    \begin{align*}
        &P_1=\sup_{g\in G}\#\{s\in \mathcal{P}_{c,t}: g\in A_{s,c,t}\}\\
        &P_2=\sup_{h\in G}\#\{s\in \mathcal{P}_{c,t}: h\in B_{s,c,t}\}.
    \end{align*}
    Then,
    \begin{align*}
        \norm{M_{c,t}(f)}\le \sqrt{P_1P_2}\norm{\lambda(f)}.
    \end{align*}
    If, in addition, $\mathcal{P}_{c,t}$ is finite, then
    \begin{align*}
        \norm{M_{c,t}(f)}\le |\mathcal{P}_{c,t}|\norm{\lambda(f)}.
    \end{align*}
\end{corollary}
\begin{proof}
    Define the Hilbert space $\mathcal{S}=\ell^2(\mathcal{P}_{c,t})$ and the vectors
    \begin{align*}
        &x_g=\left(\mathbf{1}_{g\in A_{s,c,t}}\right)_{s\in \mathcal{P}_{c,t}}\in \mathcal{S}\\
        &y_g=\left(\mathbf{1}_{g\in B_{s,c,t}}\right)_{s\in \mathcal{P}_{c,t}}\in \mathcal{S}.
    \end{align*}
    Define the scalar matrix $R=(R_{g,h})_{g,h\in G}$ by $R_{g,h}=\langle x_g,y_h\rangle$. By assumption, $R_{g,h}=(S_{c,t})_{g,h}$ whenever $gh^{-1}\in S_r$. Since $\lambda(f)_{g,h}=0$ outside this band, we have $M_{c,t}(f)=R\circ\lambda(f)$. Moreover,
    \begin{align*}
        &\sup_{g\in G}\norm{x_g}^2=\sup_{g\in G}\#\{s\in \mathcal{P}_{c,t}: g\in A_{s,c,t}\}=P_1,\\
        &\sup_{h\in G}\norm{y_h}^2=\sup_{h\in G}\#\{s\in \mathcal{P}_{c,t}: h\in B_{s,c,t}\}=P_2.
    \end{align*}
    Hence $\norm{R}_{\operatorname{cb}}\le\sqrt{P_1P_2}$, which proves the first estimate. If $\mathcal{P}_{c,t}$ is finite, then $P_1,P_2\le|\mathcal{P}_{c,t}|$, and the second estimate follows.
\end{proof}

\section{Median metric spaces}\label{section: median spaces}
Let $(X,d)$ be a metric space. Recall that $X$ is assumed to be a connected, uniformly locally finite graph equipped with its graph metric. For $x,y\in X$, define the metric interval
\begin{align*}
    [x,y]:=\{m\in X: d(x,y)=d(x,m)+d(m,y)\}.
\end{align*}
The metric space $(X,d)$ is \textit{modular} if $[x,y]\cap[y,z]\cap[z,x]\ne\varnothing$ for all $x,y,z\in X$. It is \textit{median} if this intersection consists of a single point, denoted by $m(x,y,z)$. Whenever quantitative control of the median is needed, we assume that $m:X^3\to X$ is $K$-Lipschitz in each variable for some $K\ge1$ \cite[Assumption (L2)]{SpakulaWright2017propertyA}.

Suppose that $(X,d)$ is modular, fix $x_0\in X$, and denote the sphere of radius $r$ about $x_0$ by $S_r(X)$. Let $G$ be a finitely generated group acting by isometries on $X$. A modular centroid map $\mathfrak{c}:G\times G\to X$ is obtained by choosing 
\begin{align}\label{equation: centroid in modular cases}
    \mathfrak{c}(g,h)\in [x_0,g\cdot x_0]\cap [x_0,h\cdot x_0]\cap [g\cdot x_0,h\cdot x_0].
\end{align}
We equip $G$ with the length function $\ell(g)=d(x_0,g\cdot x_0)$. We define a parametrization family $(\theta,\Theta)$ for $(G,\mathfrak{c})$ simply by setting $\Theta_r=\{0,\ldots, r\}$ and
\begin{equation*}
    \begin{aligned}
        \theta_g \colon \mathcal{C}_g &\longrightarrow \Theta_{\ell(g)} \\
                u &\longmapsto d(x_0,u).
    \end{aligned}
\end{equation*}
This map is well-defined: if $u=\mathfrak{c}(g,h)$ for some $h\in G$, then $u\in[x_0,g\cdot x_0]$, and hence
\begin{align*}
    \ell(g)=d(x_0,g\cdot x_0)=d(x_0,u)+d(u,g\cdot x_0)\ge d(x_0,u),
\end{align*}
so that $d(x_0,u)\le \ell(g)$.

The main result of this section is the following.
\begin{theorem}\label{theorem: median case}
    Let $(X,d)$ be a median metric space with base point $x_0$, and let $G$ be a finitely generated group acting by isometries on $X$. Using the above parametrization family and median map, we have
    \begin{align*}
        \sup_{(c,t)\in X\times \{0,\ldots, r\}}\norm{M_{c,t}(f)}\le \norm{\lambda(f)}\le \sup_{(c,t)\in X\times\{0,\ldots, r\}}\norm{M_{c,t}(f)}\sum_{t=0}^r\sup_{y\in X}|[x_0,y]\cap S_{t}(X)|,
    \end{align*}
    for every $f:G\to B(\Hil)$ supported in $S_r$. Moreover, only modularity of $(X,d)$ is needed for the upper bound.
\end{theorem}
We next reformulate Theorem~\ref{theorem: median case} in terms of Buchholz's spherical operators:
\begin{align*}
    (B_{i,j}(f))_{g,h}=\mathbf{1}_{g\cdot x_0\in S_i(X)}\mathbf{1}_{h\cdot x_0\in S_j(X)}f(gh^{-1}),\qquad gh^{-1}\in S_r,
\end{align*}
for $i,j\in \N$.
\begin{corollary}\label{corollary: Buchholz median case}
    Let $(X,d)$ be a median metric space with base point $x_0$, and let $G$ be a finitely generated group acting by isometries on $X$. Assume that $G$ acts transitively on $X$.
    Using the above parametrization family and median map, we have
    \begin{align*}
        \max_{t\in \{0,\ldots, r\}}\norm{B_{t,r-t}(f)}\le \norm{\lambda(f)}\le \max_{t\in \{0,\ldots, r\}}\norm{B_{t,r-t}(f)}\sum_{t=0}^r\sup_{y\in X}|[x_0,y]\cap S_{t}(X)|,
    \end{align*}
    for every $f:G\to B(\Hil)$ supported in $S_r$.
\end{corollary}

We begin the proof of Theorem~\ref{theorem: median case} with two inclusions for the sets $A_t(g)$ and $B_t(h)$ from \eqref{equation: UV}.
\begin{lemma}\label{lemma: inclusions of A and B}
    Suppose $(X,d)$ is modular.
    Fix $t\in\{0,\ldots,r\}$. If $g\in G$ satisfies $\ell(g)\ge t$, then
    \begin{align*}
        A_t(g)\subseteq [x_0,g^{-1}\cdot x_0]\cap S_{\ell(g)-t}(X).
    \end{align*}
    If $h\in G$ satisfies $\ell(h)\ge r-t$, then 
    \begin{align*}
        B_t(h)\subseteq [x_0,h^{-1}\cdot x_0]\cap S_{\ell(h)-(r-t)}(X).
    \end{align*}
\end{lemma}
\begin{proof}
    By definition,
    \begin{align*}
        A_t(g)=\{g^{-1}\cdot \mathfrak{c}(p,g): \ell(p)=r;\theta_p(\mathfrak{c}(p,g))=t\}.
    \end{align*}
    Fix $g\in G$ and choose $p\in S_r$ with $\theta_p(\mathfrak{c}(p,g))=t$. Set $u=\mathfrak{c}(p,g)$. Then $u\in [x_0,p\cdot x_0]\cap [x_0,g\cdot x_0]$. Since $G$ acts by isometries, it follows that
    \begin{align*}
        d(x_0,g\cdot x_0)=d(x_0,u)+d(u,g\cdot x_0)=d(g^{-1}\cdot x_0,g^{-1}\cdot u)+d(x_0,g^{-1}\cdot u).
    \end{align*}
    Hence $g^{-1}\cdot u\in [x_0,g^{-1}\cdot x_0]$. Since $\theta_p(u)=d(x_0,u)=t$, we also have
    \begin{align*}
        \ell(g)=t+d(x_0,g^{-1}\cdot u).
    \end{align*}
    Therefore $g^{-1}\cdot\mathfrak{c}(p,g)=g^{-1}\cdot u\in [x_0,g^{-1}\cdot x_0]\cap S_{\ell(g)-t}(X)$.
    
    For the second inclusion, recall that
    \begin{align*}
        B_t(h)=\{h^{-1}p^{-1}\cdot \mathfrak{c}(p,ph):\ell(p)=r;\theta_p(\mathfrak{c}(p,ph))=t\}.
    \end{align*}
    First observe that $[p\cdot x_0,ph\cdot x_0]=p\cdot[x_0,h\cdot x_0]$. Indeed, for $m\in X$, we have $m\in[p\cdot x_0,ph\cdot x_0]$ if and only if
        \begin{align*}
            d(x_0,h\cdot x_0)&=d(p\cdot x_0,ph\cdot x_0)\\
            &=d(p\cdot x_0,m)+d(m,ph\cdot x_0)\\
            &=d(x_0,p^{-1}\cdot m)+d(p^{-1}\cdot m,h\cdot x_0).
        \end{align*}
    Thus $m\in[p\cdot x_0,ph\cdot x_0]$ if and only if $p^{-1}\cdot m\in[x_0,h\cdot x_0]$. Setting $m=\mathfrak{c}(p,ph)$ gives 
    \begin{align*}
        d(x_0,h\cdot x_0)&=d(x_0, p^{-1}\cdot m)+d(p^{-1}\cdot m, h\cdot x_0)\\
        &=d(h^{-1}\cdot x_0, h^{-1}p^{-1}\cdot m)+d(x_0,h^{-1}p^{-1}\cdot m).
    \end{align*}
    Hence $h^{-1}p^{-1}\cdot m\in [x_0,h^{-1}\cdot x_0]$. Moreover,
    \begin{align*}
        d(x_0,p^{-1}\cdot m)=d(p\cdot x_0,m).
    \end{align*}
    Since $m\in [x_0,p\cdot x_0]$, we get
    \begin{align*}
        d(x_0,p^{-1}\cdot m)=d(x_0,p\cdot x_0)-d(x_0,m)=r-t,
    \end{align*}
    as $p\in S_r$ and $t=\theta_p(m)=d(x_0,m)$. Therefore,
    \begin{align*}
        \ell(h)=d(x_0,h\cdot x_0)=r-t+d(x_0,h^{-1}p^{-1}\cdot m).
    \end{align*}
    Thus $h^{-1}p^{-1}\cdot m=h^{-1}p^{-1}\cdot \mathfrak{c}(p,ph)\in [x_0,h^{-1}\cdot x_0]\cap S_{\ell(h)-(r-t)}(X)$.
\end{proof}
\begin{lemma}\label{lemma: lower bound median}
    Suppose $(X,d)$ is median.
    Let $f:G\to B(\Hil)$ be supported in $S_r$. Define, for $c\in X$ and $t\in \{0,\ldots, r\}$,
    \begin{align*}
        &G_{c,t}=\{g\in G: c\in [x_0,g^{-1}\cdot x_0], g\cdot c\in S_t(X)\}.
    \end{align*}
    Then, for every $c\in X$ and $t\in \Theta_r$,
    \begin{align*}
        (M_{c,t}(f))_{g,h}= \mathbf{1}_{g\in G_{c,t}}\mathbf{1}_{h\in G_{c,r-t}}f(gh^{-1}).
    \end{align*}
    Moreover,
    \begin{align*}
        \norm{M_{c,t}(f)}\le \norm{\lambda(f)},
    \end{align*}
    for every $(c,t)\in X\times \Theta_r$.
\end{lemma}
\begin{proof}
    The multiplier $M_{c,t}=S_{c,t}\circ \lambda(f)$ satisfies
    \begin{align*}
        (S_{c,t})_{g,h}=\mathbf{1}_{c=g^{-1}\cdot\mathfrak{c}(gh^{-1},g)}\mathbf{1}_{\theta_{gh^{-1}}(\mathfrak{c}(gh^{-1},g))=t},\qquad gh^{-1}\in S_r.
    \end{align*}
    Fix $g,h\in G$, $c\in X$ and $t\in \Theta_r$ such that $gh^{-1}\in S_r$. We will show that the following are equivalent. 
    \begin{enumerate}
        \item We have $(S_{c,t})_{g,h}=1$.
        \item We have $c\in A_t(g)\cap B_t(h)$.
        \item We have $g\in G_{c,t}$ and $h\in G_{c,r-t}$.
    \end{enumerate}
    By Lemma~\ref{lemma: decomposition of lambda f}, $(1)$ implies $(2)$.
    Suppose that $c\in A_t(g)\cap B_t(h)$. Lemma~\ref{lemma: inclusions of A and B} gives
    \begin{align*}
        c\in [x_0,g^{-1}\cdot x_0]\cap S_{\ell(g)-t}(X).
    \end{align*}
    Consequently,
    \begin{align*}
        d(x_0,g\cdot c)=d(c,g^{-1}\cdot x_0)=d(x_0,g^{-1}\cdot x_0)-d(x_0,c)=\ell(g)-(\ell(g)-t)=t.
    \end{align*}
    Thus $g\in G_{c,t}$, and the same argument gives $h\in G_{c,r-t}$. Hence $(2)$ implies $(3)$.
    Conversely, suppose that $g\in G_{c,t}$, $h\in G_{c,r-t}$, and $gh^{-1}\in S_r$. We show that $(S_{c,t})_{g,h}=1$. Set $p=gh^{-1}$, so that $\ell(p)=r$. By the definition of $G_{c,t}$,
    \begin{align*}
        c\in [x_0,g^{-1}\cdot x_0] \Leftrightarrow g\cdot c\in [x_0,g\cdot x_0].
    \end{align*}
    Combined with $h\in G_{c,r-t}$, we have
    \begin{align*}
        g\cdot c=ph\cdot c\in [x_0,ph\cdot x_0]\cap [p\cdot x_0, ph\cdot x_0].
    \end{align*}
    As $t=d(x_0,g\cdot c)$ and $r-t=d(x_0,h\cdot c)$, we get
    \begin{align*}
        r=d(x_0,g\cdot c)+d(x_0,h\cdot c)=d(x_0,ph\cdot c)+d(x_0,h\cdot c).
    \end{align*}
    Since $r=d(x_0,p\cdot x_0)$, we get
    \begin{align*}
        d(x_0,p\cdot x_0)=d(x_0,ph\cdot c)+d(x_0,h\cdot c)=d(x_0,ph\cdot c)+d(ph\cdot c,p\cdot x_0).
    \end{align*}
    Therefore $ph\cdot c\in [x_0,p\cdot x_0]$. That is,
    \begin{align*}
        g\cdot c=ph\cdot c\in [x_0,ph\cdot x_0]\cap [x_0,p\cdot x_0]\cap [p\cdot x_0, ph\cdot x_0].
    \end{align*}
    By uniqueness of the median,
    \begin{align*}
        g\cdot c=ph\cdot c=\mathfrak{c}(p,ph)=\mathfrak{c}(p,g).
    \end{align*}
    Note that $\theta_p(g\cdot c)=d(x_0,g\cdot c)=t$, as $g\in G_{c,t}$. It follows that
    \begin{align*}
        c=g^{-1}\cdot \mathfrak{c}(p,g)=g^{-1}\cdot \mathfrak{c}(gh^{-1},g),
    \end{align*}
    and $(3)$ implies $(1)$.
    Hence,
    \begin{align*}
        (M_{c,t}(f))_{g,h}=\mathbf{1}_{g\in G_{c,t}}\mathbf{1}_{h\in G_{c,r-t}}f(gh^{-1}).
    \end{align*}
    Corollary~\ref{corollary: lower bound simple version} implies the norm bound $\norm{M_{c,t}(f)}\le \norm{\lambda(f)}$.
\end{proof}
We now prove Theorem~\ref{theorem: median case}.
\begin{proof}[Proof of Theorem~\ref{theorem: median case}]
    Lemma~\ref{lemma: lower bound median} gives
    \begin{align*}
        \sup_{(c,t)\in X\times \{0,\ldots,r\}}\norm{M_{c,t}(f)}\le \norm{\lambda(f)}.
    \end{align*}
    Proposition~\ref{proposition: main decomposition-Haagerup} gives
    \begin{align*}
        \norm{\lambda(f)}\le \sup_{(c,t)\in X\times \{0,\ldots,r\}}\norm{M_{c,t}(f)}\sum_{t=0}^r\norm{U_t}\norm{V_t}.
    \end{align*}
    Combining Lemma~\ref{lemma: bound on norm of U_t and V_t} with Lemma~\ref{lemma: inclusions of A and B} (which only requires modularity), we get
    \begin{align*}
         \sum_{t=0}^r\norm{U_t}\norm{V_t}\le \sum_{t=0}^r \sup_{g,h\in G}\sqrt{|[x_0,g^{-1}\cdot x_0]\cap S_{\ell(g)-t}(X)||[x_0,h^{-1}\cdot x_0]\cap S_{\ell(h)-(r-t)}(X)|}.
    \end{align*}
    Applying the Cauchy--Schwarz inequality and then replacing $t$ by $r-t$ gives
    \begin{align*}
         \norm{\lambda(f)}\le \sup_{(c,t)\in X\times \{0,\ldots,r\}}\norm{M_{c,t}(f)}\sum_{t=0}^r \sup_{g\in G}|[x_0,g^{-1}\cdot x_0]\cap S_{\ell(g)-t}(X)|.
    \end{align*}
    Finally, observe that $c\in [x_0,g^{-1}\cdot x_0]\cap S_{\ell(g)-t}(X)$ if and only if
    \begin{align*}
        \ell(g)=d(x_0,g^{-1}\cdot x_0)=d(x_0,c)+d(c,g^{-1}\cdot x_0)=\ell(g)-t+d(g\cdot c,x_0).
    \end{align*}
    Then,
    \begin{align*}
        \ell(g)=d(x_0,g\cdot x_0)=d(x_0,g\cdot c)+d(g\cdot c,g\cdot x_0).
    \end{align*}
    That is, $c\in [x_0,g^{-1}\cdot x_0]\cap S_{\ell(g)-t}(X)$ if and only if $g\cdot c\in [x_0,g\cdot x_0]\cap S_t(X)$. The result follows by taking the supremum over the larger set $X\supseteq G\cdot x_0$.
\end{proof}
\begin{proof}[Proof of Corollary~\ref{corollary: Buchholz median case}]
    It suffices to show that
    \begin{align*}
        \sup_{(c,t)\in X\times \{0,\ldots,r\}}\norm{M_{c,t}(f)}=\max_{t\in \{0,\ldots, r\}}\norm{B_{t,r-t}(f)}.
    \end{align*}
    By Lemma~\ref{lemma: lower bound median}, we have
    \begin{align*}
        (M_{c,t}(f))_{g,h}=\mathbf{1}_{g\in G_{c,t}}\mathbf{1}_{h\in G_{c,r-t}}f(gh^{-1}),\qquad gh^{-1}\in S_r.
    \end{align*}
    Choose $c=x_0$. Then
    \begin{align*}
        G_{x_0,t}=\{g\in G: x_0\in [x_0,g^{-1}\cdot x_0], g\cdot x_0\in S_t(X)\}.
    \end{align*}
    Since $x_0\in [x_0,g^{-1}\cdot x_0]$,
    \begin{align*}
        G_{x_0,t}=\{g\in G: g\cdot x_0\in S_t(X)\}.
    \end{align*}
    In particular,
    \begin{align*}
        (M_{x_0,t}(f))_{g,h}=\mathbf{1}_{g\cdot x_0\in S_t(X)}\mathbf{1}_{h\cdot x_0\in S_{r-t}(X)}f(gh^{-1}),\qquad gh^{-1}\in S_r.
    \end{align*}
    Thus $M_{x_0,t}(f)=B_{t,r-t}(f)$, and therefore
    \begin{align*}
        \max_{t\in \{0,\ldots, r\}}\norm{B_{t,r-t}(f)}\le \sup_{(c,t)\in X\times \{0,\ldots, r\}}\norm{M_{c,t}(f)}.
    \end{align*}
    For the reverse inequality, observe that
    \begin{align*}
        \mathbf{1}_{g\in G_{c,t}}\mathbf{1}_{h\in G_{c,r-t}}=\mathbf{1}_{g\in G_{c,t}}\mathbf{1}_{h\in G_{c,r-t}}\mathbf{1}_{g\cdot c\in S_t(X)}\mathbf{1}_{h\cdot c\in S_{r-t}(X)},
    \end{align*}
    because membership in $G_{c,t}$ and $G_{c,r-t}$ already imposes the last two conditions. Hence
    \begin{align*}
         (M_{c,t}(f))_{g,h}=\mathbf{1}_{g\in G_{c,t}}\mathbf{1}_{h\in G_{c,r-t}} \left(\mathbf{1}_{g\cdot c\in S_t(X)}\mathbf{1}_{h\cdot c\in S_{r-t}(X)} f(gh^{-1})\right).
    \end{align*}
    Since $G$ acts transitively on $X$, choose $p\in G$ with $c=p\cdot x_0$. Then
    \begin{align*}
        (M_{c,t}(f))_{g,h}=\mathbf{1}_{g\in G_{c,t}}\mathbf{1}_{h\in G_{c,r-t}} \left(\mathbf{1}_{gp\cdot x_0\in S_t(X)}\mathbf{1}_{hp\cdot x_0\in S_{r-t}(X)} f(gh^{-1})\right).
    \end{align*}
    It follows again by Corollary~\ref{corollary: lower bound simple version} that
    \begin{align*}
        \norm{M_{c,t}(f)}\le \norm{B_{t,r-t}^{(p)}(f)}=\norm{B_{t,r-t}(f)},
    \end{align*}
    where $B_{t,r-t}^{(p)}$ is the block obtained by right translation by $p$ and is unitarily equivalent to $B_{t,r-t}$. Hence
    \begin{align*}
        \sup_{(c,t)\in X\times \{0,\ldots, r\}}\norm{M_{c,t}(f)}\le \max_{t\in \{0,\ldots, r\}}\norm{B_{t,r-t}(f)}.
    \end{align*}
\end{proof}

Our first application concerns $\CAT$ cube complexes; see \cite{Schwer2023CATintroduction} for background and notation. Let $X$ be a $\CAT$ cube complex of rank $\nu$, and denote its $0$- and $1$-skeleta by $X^0$ and $X^1$, respectively. It is well known that $X^1$ is a median metric space; see \cite[Section 2]{Bowditch2014embedding}.
\begin{corollary}\label{corollary: CAT bounds}
    Let $X$ be a $\CAT$ cube complex of rank $\nu$ and let $x_0\in X$ be a base point. Suppose that $G$ acts freely on $X^0$. Then, for every $f:G\to B(\Hil)$ supported in $S_r$,
    \begin{align*}
        \sup_{(c,t)\in X^0\times \{0,\ldots, r\}}\norm{M_{c,t}(f)}\le \norm{\lambda(f)}\le \binom{\nu+r}{r}\sup_{(c,t)\in X^0\times \{0,\ldots, r\}}\norm{M_{c,t}(f)}.
    \end{align*}
    Moreover, if $G$ acts transitively (and hence regularly) on $X^0$,
    \begin{align*}
        \max_{t\in \{0,\ldots, r\}}\norm{B_{t,r-t}(f)}\le \norm{\lambda(f)}\le \binom{\nu+r}{r}\max_{t\in \{0,\ldots, r\}}\norm{B_{t,r-t}(f)}.
    \end{align*}
\end{corollary}
Theorem~\ref{theorem: median groups} follows from Corollary~\ref{corollary: CAT bounds}, since $G$ acts regularly on its Cayley graph $X=(G, S)$ and every regular action is isomorphic to the left-regular action. 
\begin{proof}[Proof of Corollary~\ref{corollary: CAT bounds}]
    We may assume that $\nu\ge1$, since the rank-zero case is trivial.
    By Theorem~\ref{theorem: median case}, it suffices to show that
    \begin{align*}
        \sum_{t=0}^r\sup_{y\in X^0}|[x_0,y]\cap S_t(X^0)| \le \binom{\nu+r}{r}.
    \end{align*}
    The proof of \cite[Theorem 0.4]{ChatterjiRuane2005median} shows that
    \begin{align*}
        \sup_{y\in X^0}|[x_0,y]\cap S_t(X^0)|\le P_{t,r}(\nu),
    \end{align*}
    where $P_{t,r}(\nu)$ is the number of vertices $v$ occurring in geodesic paths from two opposite vertices $a,b$ of the $\nu$-cube $\{0,\ldots, r\}^\nu$ whose distance to $a$ is $t$. In particular, a stars-and-bars argument yields
    \begin{align*}
        P_{t,r}(\nu)=\binom{\nu+t-1}{\nu-1}.
    \end{align*}
    Summing over $t=0,\ldots,r$ gives
    \begin{align*}
        \sum_{t=0}^r\sup_{y\in X^0}|[x_0,y]\cap S_t(X^0)|\le \sum_{t=0}^r\binom{\nu+t-1}{\nu-1}=\binom{\nu+r}{\nu},
    \end{align*}
    where the last equality is the hockey-stick identity.
\end{proof}
Right-angled Artin groups form an important class of groups acting on $\CAT$ cube complexes. Let $\Gamma=(V,E)$ be a finite graph. We define its right-angled Artin group $G(\Gamma)$ by the presentation
\begin{align*}
    G(\Gamma)=\langle g_v: v\in V\mid [g_u,g_v]=1\text{ for every }(u,v)\in E \rangle.
\end{align*}
$G(\Gamma)$ acts regularly on the universal cover of the Salvetti complex \cite[Section 4.4]{Schwer2023CATintroduction}. The Salvetti complex has rank equal to the largest complete subgraph of $\Gamma$, that is, the clique number $\omega(\Gamma)$.
We therefore obtain the following consequence.
\begin{corollary}\label{corollary: right angled artin groups}
    Let $\Gamma$ be a finite graph with clique number $\omega(\Gamma)$. Then, for every $f:G(\Gamma)\to B(\Hil)$ supported in $S_r$,
    \begin{align*}
         \max_{t\in \{0,\ldots, r\}}\norm{B_{t,r-t}(f)}\le \norm{\lambda(f)}\le \binom{\omega(\Gamma)+r}{r}\max_{t\in \{0,\ldots, r\}}\norm{B_{t,r-t}(f)}.
    \end{align*}
\end{corollary}

\section{Coarse median groups}\label{section: coarse median spaces}
We now extend the preceding results to coarse median spaces \cite{Bowditch2013coarsemedian}. We use the algebraic formulation of median structures; for $\CAT$ cube complexes equipped with their natural metric, this formulation agrees with the metric one \cite[Section 2.2]{NibloWrightZang2019fourpoint}.

A median space $(Y,m)$ is a set $Y$ equipped with a ternary map $m:Y^3\to Y$ satisfying 
\begin{enumerate}
    \item $m(x,y,z)=m(y,z,x)=m(y,x,z)$.
    \item $m(x,x,y)=x$.
    \item $m(x,y,m(z,a,b))=m(m(x,y,z),m(x,y,a),b)$.
\end{enumerate}
For a median space $(Y,m)$, define the intervals
\begin{align*}
    [x,y]=\{z\in Y: m(x,z,y)=z\}=\{m(x,z,y):z\in Y\}.
\end{align*}
A median metric space $(Y,d,m)$ is a metric space $(Y,d)$ for which $(Y,m)$ is a median space. As before, we assume $(Y,d)$ is a connected, uniformly locally finite graph endowed with its graph metric.
We assume that $m$ is $K$-Lipschitz for some $K\ge1$. Fix a base point $x_0\in Y$. We recall the iterated medians from \cite[Definition 5.1]{SpakulaWright2017propertyA}. Let $(x_i)_{i\in \N}\subseteq Y$ and $b\in Y$.
Define the iterated medians recursively by
\begin{align*}
    &m(x_1;b):=x_1\\
    &m(x_1,\ldots,x_{n+1};b):=m(m(x_1,\ldots, x_n;b),x_{n+1},b);\quad n\ge 1.
\end{align*}
By \cite[Lemma 5.2]{SpakulaWright2017propertyA}, the map $(x_1,\ldots,x_n)\mapsto m(x_1,\ldots,x_n;b)$ is symmetric. Hence, for a finite set $A\subseteq Y$, the following definition is independent of the ordering of $A$:
\begin{align*}
    q_{A,b}:=m((x)_{x\in A};b).
\end{align*}
Equivalently, $q_{A,b}$ is the unique point $q$ satisfying
\begin{align}\label{equation: q characterization as intersection}
    \bigcap_{x\in A}[b,x]=[b,q].
\end{align}
For $x\in Y$ and $r\in\N$, let $B_r(x)=\{z\in Y:d(x,z)\le r\}$ and define
\begin{equation}
    \label{equation: map q median case}
    \begin{aligned}
        q \colon Y\times \N &\longrightarrow Y \\
                (x,r) &\longmapsto q_{B_r(x),x_0}.
    \end{aligned}
\end{equation}
For $r\in\N$, write $q_r(\cdot)=q(\cdot,r)$ and call $q_r$ the projection map.
\begin{lemma}\label{lemma: distance of meets}
    Let $(Y,d,m)$ be a median metric space of finite rank $\nu$ and Lipschitz constant $K\ge 1$, and let $x_0\in Y$ be a base point. Let $q: Y\times \N\to Y$ be defined as in \eqref{equation: map q median case} and let $q_r$ be the projection map.
    Then the following statements hold.
    \begin{enumerate}
        \item For every $(x,r)\in Y\times \N$ and $z\in B_r(x)$, we have $q_r(x)\in [x_0,z]$.
        \item There exists a constant $K_{\nu}\ge 1$, depending only on $(K,\nu)$, such that $d(x,q_r(x))\le K_\nu r$ for every $(x,r)\in Y\times \N$.
    \end{enumerate}
\end{lemma}
\begin{proof}
    The result follows from the arguments in \cite[Section 5]{SpakulaWright2017propertyA}. First, \eqref{equation: q characterization as intersection} yields
    \begin{align*}
        \bigcap_{z\in B_r(x)} [z,x_0]=[q_r(x),x_0].
    \end{align*}
    Since $q_r(x)\in[q_r(x),x_0]$, it follows that $q_r(x)\in[x_0,z]$ for every $z\in B_r(x)$. By \cite[Lemma 5.3.iii]{SpakulaWright2017propertyA}, there exist $d\le \nu$ and $z_1,\ldots, z_d\in B_r(x)$ such that
    \begin{align*}
        q_r(x)=m(z_1,\ldots, z_d;x_0).
    \end{align*}
    By \cite[Proposition 5.5]{SpakulaWright2017propertyA}, there exists a constant $K_d:=3^dK^d$ such that
    \begin{align*}
        d(x,q_r(x))=d(x,m(z_1,\ldots, z_d;x_0)) \le 3^dK^d \max_{1\le i\le d}d(x,z_i)\le K_\nu r.
    \end{align*}
\end{proof}
We now describe the coarse median setting. A connected, uniformly locally finite metric graph $(X,d)$ equipped with a ternary map $\mu: X^3\to X$ is a \textit{coarse median space} if there exist $\rho\in\R_+$ and $h:\N\to\R_+$ satisfying the following conditions.
\begin{enumerate}
    \item $\mu$ is uniformly $\rho$-Lipschitz up to the additive constant $h(0)$: for all $x,y,z,x',y',z'\in X$,
    \begin{align*}
        d(\mu(x,y,z),\mu(x', y',z'))\le \rho(d(x,x')+d(y, y')+d(z,z'))+h(0).
    \end{align*}
    \item\label{property: coarse median property} For every $p\in\N$ and every $Y\subseteq X$ with $|Y|\le p$, there exist a finite $\CAT$ cube complex $\Pi$ with median operation $m$, together with maps $E:Y\to\Pi$ and $\pi:\Pi\to X$, such that, for all $x,y,z\in\Pi$,
    \begin{align*}
        d\left(\pi(m(x,y,z)), \mu(\pi(x),\pi(y),\pi(z))\right)\le h(p),
    \end{align*}
    and
    \begin{align*}
        d\left(a, \pi\circ E(a)\right)\le h(p),
    \end{align*}
    for every $a\in Y$.
    \item\label{property: median aspect} For all $x,y,z\in X$, we have $\mu(x,x,y)=x$ and $\mu(x,y,z)=\mu(y,z,x)=\mu(y,x,z)$.
\end{enumerate}
We call $(\rho,h)$ the parameters of $(X,d,\mu)$. Condition~\ref{property: median aspect} is not part of Bowditch's original definition, but, as in \cite{Bowditch2014embedding}, one may replace the coarse median by another at uniformly bounded distance so that this condition holds. The space has rank at most $\nu$ if the cube complexes $\Pi$ in Condition~\ref{property: coarse median property} can always be chosen with rank at most $\nu$.

Let $(X,d,\mu)$ be a coarse median space. 
For $\kappa>0$, define the coarse interval
\begin{align*}
    &[x,y]_\kappa=\{z\in X: d(z,\mu(x,y,z))\le \kappa\}.
\end{align*}

The main theorem of this section is the following.
\begin{theorem}\label{theorem: coarse median bounds}
    Let $(X,d,\mu)$ be a coarse median space of rank at most $\nu$ with base point $x_0$, and let $G$ be a finitely generated group acting isometrically on $X$. Equip $G$ with the proper orbit length $\ell(g)=d(x_0,g\cdot x_0)$. Then there exists a function $P:\N\to\R_+$ satisfying
    \begin{align*}
        P(r)=o(r^{\nu/2+\varepsilon})\qquad\text{for every }\varepsilon>0,
    \end{align*}
    such that, for every $f:G\to B(\Hil)$ supported in $S_r$,
    \begin{align*}
        \sup_{c\in  X}\norm{M_{c}(f)}\le \norm{\lambda(f)}\le P(r)\sup_{c\in X}\norm{M_{c}(f)},
    \end{align*}
    where the operators $M_c(f)$ are defined explicitly in the proof.
    Moreover, there is an affine function $L:\N\to\R_+$ such that, if the action of $G$ on $X$ is transitive, then
    \begin{align*}
        \norm{\lambda(f)}\le (2r+1)P(r)\max_{\substack{0\le i\le L(r)\\ |i-j|\le r \le i+j}}\norm{B_{i,j}(f)}.
    \end{align*}
\end{theorem}
Theorem~\ref{theorem: coarse median groups} follows from Theorem~\ref{theorem: coarse median bounds}, since $G$ acts regularly on its Cayley graph $X=(G,S)$.

The proof of Theorem~\ref{theorem: coarse median bounds} requires a coarse analogue of Lemma~\ref{lemma: distance of meets}. Let $q:X\times\N\to X$ and write $q_r(\cdot)=q(\cdot,r)$. We say that $(X,d,\mu,q)$ satisfies the \textit{coarse iterate median criteria} with respect to $x_0$ if
\begin{enumerate}[label=(M\arabic*)]
    \item\label{property: m1} There exists a constant $\kappa>0$ such that for all $(x,r)\in X\times \N$ and $z\in B_r(x)$, we have $q_r(x)\in [x_0,z]_\kappa$.
    \item\label{property: m2} There exists an affine function $L:\N\to \N$ such that for all $(x,r)\in X\times \N$, we have $d(x,q_r(x))\le L(r)$.
\end{enumerate}
When these conditions hold, we call $(\kappa, L)$ the parameters of $q$. The following proposition produces such a map with uniform parameters.
\begin{proposition}\label{proposition: coarse meet existence}
    Let $(X,d,\mu)$ be a coarse median space of rank at most $\nu$, and let $x_0\in X$ be a base point. Then there exists a map $q: X\times \N\to X$ satisfying the coarse iterate median criteria, with parameters depending only on $(\rho,h)$ and $\nu$.
\end{proposition}
We first deduce Theorem~\ref{theorem: coarse median bounds} from this proposition.
\begin{proof}[Proof of Theorem~\ref{theorem: coarse median bounds}]
    Let $q$ be provided by Proposition~\ref{proposition: coarse meet existence}, and define the centroid map $\mathfrak{c}:G^2\to X$ by
    \begin{align*}
        \mathfrak{c}(p,g)=g\cdot q_{\ell(p)}(g^{-1}\cdot x_0).
    \end{align*}
    Set $\Theta_r=\{*\}$ and let every parametrization map $\theta_g$ be constant. We may assume that $S_r\ne\varnothing$, since otherwise the statement is trivial.
    In this case, we write
    \begin{align*}
        &A_*(g)=\{g^{-1}\cdot \mathfrak{c}(p,g): \ell(p)=r\}=\{q_r(g^{-1}\cdot x_0)\}.\\
        &B_*(h)=\{h^{-1}p^{-1}\cdot \mathfrak{c}(p,ph): \ell(p)=r\}=\{q_r(h^{-1}p^{-1}\cdot x_0):\ell(p)=r\}.
    \end{align*}
    The corresponding block entries are
    \begin{align*}
        (M_{c,*}(f))_{g,h}=\mathbf{1}_{c=q_r(g^{-1}\cdot x_0)}f(gh^{-1}),\qquad gh^{-1}\in S_r.
    \end{align*}
    We suppress the index $*$ from now on. Since $|A(g)|\le1$ for every $g\in G$, Lemma~\ref{lemma: bound on norm of U_t and V_t} gives $\norm{U}\le1$. Fix $p\in S_r$ and set $x=h^{-1}p^{-1}\cdot x_0$ and $z=h^{-1}\cdot x_0$. Because the action is isometric, $d(x,z)=\ell(p)=r$, so $z\in B_r(x)$. Property~\ref{property: m1} gives 
    \begin{align*}
        q_r(x)\in [x_0,z]_\kappa.
    \end{align*}
    Therefore,
    \begin{align*}
        B(h)\subseteq [x_0,z]_\kappa=[x_0,h^{-1}\cdot x_0]_\kappa.
    \end{align*}
    By the triangle inequality and Property~\ref{property: m2}, for every $p\in S_r$,
    \begin{align*}
        d(h^{-1}\cdot x_0, q_r(h^{-1}p^{-1}\cdot x_0))&\le d(h^{-1}\cdot x_0, h^{-1}p^{-1}\cdot x_0)+d(h^{-1}p^{-1}\cdot x_0, q_r(h^{-1}p^{-1}\cdot x_0))\\
        &=r+L(r).
    \end{align*}
    Hence $B(h)$ has diameter at most $2\widetilde{L}(r)$, where $\widetilde{L}(r)=r+L(r)$. By \cite[Proposition 9.8]{Bowditch2014embedding}, there is a function $D:\N\to\N$ such that
    \begin{align*}
        D(r)=o(r^{\nu+\varepsilon})\qquad\text{for every }\varepsilon>0,
        \qquad |B(h)|\le D(r).
    \end{align*}
    Proposition~\ref{proposition: main decomposition-Haagerup} therefore yields
    \begin{align*}
        \norm{\lambda(f)}\le \sqrt{D(r)}\sup_{c\in X}\norm{M_c(f)}.
    \end{align*}
    Setting $P(r)=\sqrt{D(r)}$ gives $P(r)=o(r^{\nu/2+\varepsilon})$ for every $\varepsilon>0$.
    The Schur multiplier $\mathbf{1}_{c=q_r(g^{-1}\cdot x_0)}$ depends only on the row index $g$. Hence Corollary~\ref{corollary: lower bound simple version} gives
    \begin{align*}
        \sup_{c\in X}\norm{M_c(f)} \le \norm{\lambda(f)}.
    \end{align*}
    This proves the first assertion of Theorem~\ref{theorem: coarse median bounds}. Now set $G_c=\{g\in G: c=q_r(g^{-1}\cdot x_0)\}$. Then
    \begin{align*}
        (M_c(f))_{g,h}=\mathbf{1}_{g\in G_c}f(gh^{-1}),\qquad gh^{-1}\in S_r.
    \end{align*}
    For $c\in X$ and $i,j\in \N$, define
    \begin{align*}
        &R_{c,i}=\{g\in G: c=q_{r}(g^{-1}\cdot x_0); d(c,g^{-1}\cdot x_0)=i\}\\
        &C_{c,j}=\{h\in G:d(c,h^{-1}\cdot x_0)=j\}.
    \end{align*}
    Let $E_{c,i},F_{c,j}$ be the orthogonal projections onto $R_{c,i}$ and $C_{c,j}$, respectively. If $(M_c(f))_{g,h}\ne 0$, set
    \begin{align*}
        i=d(c, g^{-1}\cdot x_0);\quad j=d(c,h^{-1}\cdot x_0).
    \end{align*}
    Then, by Property~\ref{property: m2}, we have
    \begin{align*}
        i=d(q_{r}(g^{-1}\cdot x_0),g^{-1}\cdot x_0) \le L(r).
    \end{align*}
    Moreover,
    \begin{align*}
        j=d(q_{r}(g^{-1}\cdot x_0), h^{-1}\cdot x_0)\le i+d(g^{-1}\cdot x_0,h^{-1}\cdot x_0)=i+r\le r+L(r).
    \end{align*}
    Similarly, $i\le j+r$.
    Finally, 
    \begin{align*}
        i+j=d(q_{r}(g^{-1}\cdot x_0), g^{-1}\cdot x_0)+d(q_{r}(g^{-1}\cdot x_0), h^{-1}\cdot x_0)\ge r.
    \end{align*}
    Let
    \begin{align*}
        C_r=\{(i,j)\in \N: i\le L(r);|i-j|\le r\le i+j\}.
    \end{align*}
    Therefore, we have
    \begin{align*}
        (M_{c}(f))_{g,h}= \mathbf{1}_{c=q_r(g^{-1}\cdot x_0)}\mathbf{1}_{(d(c,g^{-1}\cdot x_0),d(c,h^{-1}\cdot x_0))\in C_r}f(gh^{-1}),\qquad gh^{-1}\in S_r.
    \end{align*}
    The Schur multiplier $\mathbf{1}_{c=q_r(g^{-1}\cdot x_0)}$ is a row compression. Corollary~\ref{corollary: lower bound simple version} gives
    \begin{align*}
        \norm{M_c(f)}\le \norm{\left(\mathbf{1}_{(d(c,g^{-1}\cdot x_0),d(c,h^{-1}\cdot x_0))\in C_r}f(gh^{-1})\right)_{g,h\in G}}.
    \end{align*}
    Since the action is transitive, choose $p\in G$ with $p\cdot x_0=c$. Conjugating by the right-translation unitary $g\mapsto gp$ gives
    \begin{align*}
          \sup_{c\in X}\norm{M_c(f)}\le \norm{\left(\mathbf{1}_{(\ell(g),\ell(h))\in C_r}f(gh^{-1})\right)_{g,h\in G}}.
    \end{align*}
    The operator on the right can be written as
    \begin{align*}
        \mathbf{1}_{(\ell(g),\ell(h))\in C_r}f(gh^{-1})=\sum_{(i,j)\in C_r} \Id\otimes P_{i}\lambda(f) \Id\otimes P_{j},
    \end{align*}
    where $P_i$ is the projection onto $\ell^2(S_i)$. Orthogonality of these projections and the row--column degree estimate give
    \begin{align*}
        \norm{\left(\mathbf{1}_{(\ell(g),\ell(h))\in C_r}f(gh^{-1})\right)_{g,h\in G}}\le K_r\max_{(i,j)\in C_r}\norm{B_{i,j}(f)},
    \end{align*}
    where
    \begin{align*}
        K_r^2=\max_{i\in \N}\#\{j\in \N:(i,j)\in C_r\} \max_{j\in \N}\#\{i\in \N:(i,j)\in C_r\}.
    \end{align*}
    For each fixed $i$ or $j$, the condition $|i-j|\le r$ leaves at most $2r+1$ possibilities for the other index. Hence $K_r^2\le(2r+1)^2$. Combining this with the first part yields
    \begin{align*}
        \norm{\lambda(f)}\le (2r+1)P(r)\max_{(i,j)\in C_r}\norm{B_{i,j}(f)}.
    \end{align*}
\end{proof}
We prove Proposition~\ref{proposition: coarse meet existence} using the following general result.
\begin{proposition}\label{proposition: coarse approximation}
    Let $(X,d,\mu)$ be a coarse median metric space of rank at most $\nu$, with base point $x_0$. Then there are constants $E^*,K_1,K_2$ depending only on $\nu$ and the coarse parameters such that, for every $Z\subseteq X$ and $a\in X$ satisfying
    \begin{align*}
        R(a,Z):=\sup_{z\in Z}d(a,z)<\infty,
    \end{align*}
    there exists $h\in X$ with the properties
    \begin{enumerate}[label=(P\arabic*)]
        \item\label{property: inclusion with E*} $h\in [x_0,z]_{E^*}$ for every $z\in Z$.
        \item\label{property: distance bound with K_1K_2} $d(a,h)\le K_1R(a,Z)+K_2$.
    \end{enumerate}
\end{proposition}
Proposition~\ref{proposition: coarse meet existence} follows by applying Proposition~\ref{proposition: coarse approximation} with $Z=B_r(x)$ and $a=x$, and then defining $q_r(x)$ to be the resulting point $h$.

We turn to the proof of Proposition~\ref{proposition: coarse approximation}. The argument has three steps.
\begin{enumerate}
    \item We prove a coarse version of \cite[Lemma 5.4]{SpakulaWright2017propertyA}. It is the analogue of \cite[Lemma 4.8]{NibloWrightZang2019fourpoint} for the coarse intervals $[a,b]_\kappa$.
    \item We use \cite[Corollary 6.3]{SpakulaWright2017propertyA} to obtain an a priori linear coarse inclusion $\widetilde{h}\in [x_0,z]_{L(R(a,Z))}$ for all $z\in Z$.
    \item We iterate the procedure to obtain the uniform inclusion $h\in [x_0,z]_{E^*}$ for all $z\in Z$.
\end{enumerate}
We now carry out these steps. For brevity, write $\braket{x,y,z}=\mu(x,y,z)$. By \cite[Lemma 9.2]{Bowditch2014embedding}, there is a constant $\lambda>0$ such that, for all $x,y,z\in X$, 
\begin{align*}
    \braket{x,y,z}\in [x,y]_\lambda.
\end{align*}
There is also a constant $\gamma>0$ for which the coarse five-point relation holds \cite[Equation (2)]{NibloWrightZang2021}: for all $x,y,z,a,b\in X$,
\begin{align*}
    d\left(\braket{\braket{x,y,z},a,b}, \braket{\braket{a,b,x},\braket{a,b,y},z}\right)\le \gamma.
\end{align*}
This is commonly written as $\braket{\braket{x,y,z},a,b} \sim_\gamma \braket{\braket{a,b,x},\braket{a,b,y},z}$. 
We first record two inclusions for coarse intervals.
\begin{lemma}\label{lemma: inclusions coarse}
    Let $(X,d,\mu)$ be a coarse median metric space and let $\kappa_1,\kappa_2>0$. Then there are affine maps $Q_1,Q_2:\R_+^2\to\R_+$ such that, for all $x, y, a, b\in X$, the following implications hold.
    \begin{enumerate}
        \item If $x\in [a,b]_{\kappa_1}$ and $y\in [a,x]_{\kappa_2}$, we have $y\in [a,b]_{Q_1(\kappa_1,\kappa_2)}$.
        \item If $x\in [a,b]_{\kappa_1}$ and $y\in [a,x]_{\kappa_2}$, we have $x\in [y,b]_{Q_2(\kappa_1,\kappa_2)}$.
    \end{enumerate}
\end{lemma}
\begin{proof}
    For part (1), we must show that
    \begin{align*}
        d\left(y,\braket{a,b,y}\right)\le Q_1(\kappa_1,\kappa_2).
    \end{align*}
    By the triangle inequality, 
    \begin{align*}
        d\left(y,\braket{a,b,y}\right)\le d\left(y, \braket{a,b,\braket{a,x,y}}\right)+d\left(\braket{a,b,\braket{a,x,y}},\braket{a,b,y}\right).
    \end{align*}
    The second term is at most $\rho\kappa_2+h(0)$ by the coarse Lipschitz property and the assumption $y\in[a,x]_{\kappa_2}$. For the first term, another application of the triangle inequality gives
    \begin{align*}
        &d\left(y, \braket{a,b,\braket{a,x,y}}\right)\\
        &\le d\left(y, \braket{\braket{a,b,a},\braket{a,b,x},y}\right)+d\left(\braket{\braket{a,b,a},\braket{a,b,x},y},\braket{a,b,\braket{a,x,y}}\right).
    \end{align*}
    The coarse five-point relation bounds the second term by $\gamma$. As $\braket{a,b,a}=a$, we have
    \begin{align*}
        d\left(y,\braket{a,b,y}\right)\le \rho \kappa_2+h(0)+\gamma+d\left(y, \braket{a,\braket{a,b,x},y}\right).
    \end{align*}
    The triangle inequality yields
    \begin{align*}
        d\left(y, \braket{a,\braket{a,b,x},y}\right)\le d\left(y, \braket{a,x,y}\right)+d\left(\braket{a,x,y},\braket{a,\braket{a,b,x},y}\right).
    \end{align*}
    The first term is bounded by $\kappa_2$, whereas the second term is bounded by $\rho \kappa_1+h(0)$. It follows that
    \begin{align*}
         d\left(y,\braket{a,b,y}\right)\le (\rho+1)\kappa_2+\rho\kappa_1+\gamma+2h(0)=: Q_1(\kappa_1,\kappa_2).
    \end{align*}
    Part (2) is analogous; it is the coarse version of \cite[Lemma 5.4]{SpakulaWright2017propertyA}. One may take
    \begin{align*}
        Q_2(\kappa_1,\kappa_2)=\rho(\kappa_1+\kappa_2)+2h(0)+\gamma.
    \end{align*}
\end{proof}
We next recall a consequence of \cite{SpakulaWright2017propertyA}. Set
\begin{align*}
    &L_1(r):=(\rho+1)r+\rho \lambda+\gamma+2h(0)\\
    &L_2(r):=(\rho+2)r+h(0)\\
    &L_3(r,t)=K_\nu rt+r,
\end{align*}
where $K_\nu=3^\nu \rho^\nu$. 
\begin{lemma}[{\cite[Corollary 6.3]{SpakulaWright2017propertyA}}]\label{lemma: wright corollary}
    Let $(X,d,\mu)$ be a coarse median metric space of rank at most $\nu$. For every $\kappa,t>0$, there exists $r_t>0$ such that, for every $r\ge r_t$ and every $a,b\in X$, one can find $h\in X$ satisfying
    \begin{enumerate}
        \item $h\in [a,b]_{L_1(r)}$.
        \item $d(a,h)\le L_3(r,t)$.
        \item $B_{rt}(a)\cap [a,b]_\kappa \subseteq [a,h]_{L_2(r)}$.
    \end{enumerate}
\end{lemma}
For $r\in\N$, set $P(r)=Q_2(L_1(r),L_2(r))$ and 
\begin{align*}
    W(r)=Q_1(\lambda, P(r))=: Ar+B;\quad A,B\ge 0.
\end{align*}
Choose $T\ge 1$ large enough that
\begin{align*}
    \theta:=\frac{A}{T}<\frac{1}{2}.
\end{align*}
The resulting constant $\theta$ depends only on the coarse parameters and the rank $\nu$. The iteration follows from the following corollary.
\begin{corollary}\label{corollary: iterative procedure}
    There exist constants $C_1,C_2,C_3>0$, depending only on $T$, the coarse parameters, and the rank, with the following property. Let $Z\subseteq X$, $a\in X$, and $E\ge 0$ satisfy
    \begin{align*}
        a\in [x_0,z]_E\qquad\text{for every }z\in Z.
    \end{align*}
    Then there exists $h\in X$ such that
    \begin{align*}
        &h\in [x_0,z]_{\theta E+C_1}\qquad\text{for every }z\in Z,\\
        &d(a,h)\le C_2E+C_3.
    \end{align*}
\end{corollary}
\begin{proof}
    Apply Lemma~\ref{lemma: wright corollary} with $\kappa=\lambda$, $t=T$, and $b=x_0$, and let $R_T$ denote the corresponding threshold. Let
    \begin{align*}
        r=\left\lceil\max\left\{R_T, \frac{E}{T}\right\}\right\rceil,
    \end{align*}
    and let $h\in X$ be the resulting point, so that $h\in[x_0,a]_{L_1(r)}$.
    For each $z\in Z$, define $m_z=\braket{x_0,z,a}\in [x_0,z]_\lambda\cap [x_0,a]_\lambda$. Since $a\in[x_0,z]_E$,
    \begin{align*}
        d(a,m_z)\le E.
    \end{align*}
    Since $m_z\in B_{E}(a)\subseteq B_{rT}(a)$ and $m_z\in [x_0,a]_\lambda$, Lemma~\ref{lemma: wright corollary} implies that
    \begin{align*}
        m_z\in [h,a]_{L_2(r)}.
    \end{align*}
    Lemma~\ref{lemma: inclusions coarse} then gives
    \begin{align*}
        h\in [x_0,m_z]_{Q_2(L_1(r),L_2(r))}=[x_0,m_z]_{P(r)}.
    \end{align*}
    Since $m_z\in[x_0,z]_\lambda$, a second application of Lemma~\ref{lemma: inclusions coarse} gives
    \begin{align*}
        h\in [x_0,z]_{Q_1(\lambda, P(r))}=[x_0,z]_{W(r)}.
    \end{align*}
    Moreover,
    \begin{align*}
        W(r)=Ar+B\le A\left(R_T+\frac{E}{T}+1\right)+B=\theta E+AR_T+A+B.
    \end{align*}
    Choose $C_1=AR_T+A+B$ so that
    \begin{align*}
        h\in [x_0,z]_{\theta E+C_1}\qquad\text{for every }z\in Z.
    \end{align*}
    This proves the interval inclusion. For the distance estimate,
    \begin{align*}
        d(a,h)&\le K_\nu rT+r\le K_\nu T\left(R_T+\frac{E}{T}+1\right)+\left(R_T+\frac{E}{T}+1\right)\\
        &=E\left(K_\nu+\frac{1}{T}\right)+\left(R_T+1+K_\nu TR_T+K_\nu T\right).
    \end{align*}
    Thus the distance bound holds with $C_2=\left(K_\nu+\frac{1}{T}\right)$ and $C_3=R_T+1+K_\nu TR_T+K_\nu T$.
\end{proof}
We can now complete the proof of Proposition~\ref{proposition: coarse approximation}.
\begin{proof}[Proof of Proposition~\ref{proposition: coarse approximation}]
    For every $z\in Z$,
    \begin{align*}
        d(a,\braket{x_0,z,a})=d(\braket{x_0,a,a},\braket{x_0,z,a})\le\rho R(a,Z)+h(0)=:E_0.
    \end{align*}
    Thus $a\in[x_0,z]_{E_0}$ for every $z\in Z$.
    Apply Corollary~\ref{corollary: iterative procedure} with $a_0=a$ and $E=E_0$. This yields $a_1\in X$ such that
    \begin{align*}
        &a_1\in [x_0,z]_{\theta E_0+C_1}\qquad\text{for every }z\in Z,\\
        &d(a_1,a)\le C_2E_0+C_3.
    \end{align*}
    Recursively, let $a_{n+1}\in X$ be the output of Corollary~\ref{corollary: iterative procedure} applied with $a=a_n$ and $E=E_n$. Then
    \begin{equation}
    \label{equation: recursive bounds}
    \begin{aligned}
        &a_{n+1}\in [x_0,z]_{E_{n+1}}\qquad\text{for every }z\in Z,\\
        &E_{n+1}=\theta E_n+C_1\\
        &d(a_{n+1},a_{n})\le C_2 E_n+C_3.
    \end{aligned}      
    \end{equation}
    Define
    \begin{align*}
        E^*=\max \left\{1,\frac{2C_1}{1-\theta}\right\};\quad \delta=\frac{1+\theta}{2}<1.
    \end{align*}
    Note that $E^*$ and $\delta$ depend only on $\nu$ and the coarse parameters.
    Whenever $E_n>E^*$,
    \begin{align*}
        E_n>\frac{2C_1}{1-\theta},
    \end{align*}
    and therefore $C_1< \frac{1-\theta}{2}E_n$. In particular,
    \begin{align*}
        E_{n+1}=\theta E_n+C_1< \theta E_n+\frac{1-\theta}{2}E_n=\delta E_n.
    \end{align*}
    Since $\delta<1$, define
    \begin{align*}
        N=\min\{n\in \N: E_n \le E^*\}<\infty.
    \end{align*}
    Then, for every $k<N$, $E_{k+1}\le \delta E_k$.
    Set $h=a_N$.
    By \eqref{equation: recursive bounds}, we have
    \begin{align*}
        h\in [x_0,z]_{E_N}\subseteq [x_0,z]_{E^*}\qquad\text{for every }z\in Z.
    \end{align*}
    This proves Property~\ref{property: inclusion with E*}. The triangle inequality gives
    \begin{align*}
        d(a,h)\le \sum_{k=0}^{N-1} d(a_k,a_{k+1}).
    \end{align*}
    For every $0\le k<N$, $E_k>E^*$. Hence, \eqref{equation: recursive bounds} yields
    \begin{align*}
        d(a_k,a_{k+1})\le C_2 E_k+C_3=\left(C_2+\frac{C_3}{E_k}\right)E_k\le \left(C_2+\frac{C_3}{E^*}\right)E_k.
    \end{align*}
    Consequently,
    \begin{align*}
        d(a,h)\le \left(C_2+\frac{C_3}{E^*}\right) \sum_{k=0}^{N-1}E_k.
    \end{align*}
    Since $E_{k+1}\le\delta E_k$, the geometric-series estimate gives
    \begin{align*}
        d(a,h)\le \left(C_2+\frac{C_3}{E^*}\right)E_0 \sum_{k=0}^{\infty}\delta^k=\frac{1}{1-\delta}\left(C_2+\frac{C_3}{E^*}\right)E_0.
    \end{align*}
    The constants $\delta,C_2,C_3,E^*$ depend only on the coarse parameters and $\nu$, while $E_0=\rho R(a,Z)+h(0)$. This proves Property~\ref{property: distance bound with K_1K_2}.
\end{proof}



\begin{thebibliography}{10}

\bibitem{Haagerup1978originalpaper}
Uffe Haagerup.
\newblock An example of a nonnuclear {$C\sp{\ast} $}-algebra, which has the
  metric approximation property.
\newblock {\em Invent. Math.}, 50(3):279--293, 1979.

\bibitem{Jolissaint1990RD}
Paul Jolissaint.
\newblock Rapidly decreasing functions in reduced {$C^*$}-algebras of groups.
\newblock {\em Trans. Amer. Math. Soc.}, 317(1):167--196, 1990.

\bibitem{DrutuSapirrelativehyperbolic}
Cornelia Dru\c tu and Mark Sapir.
\newblock Relatively hyperbolic groups with rapid decay property.
\newblock {\em Int. Math. Res. Not.}, (19):1181--1194, 2005.

\bibitem{Chatterji2017Introduction}
Indira Chatterji.
\newblock Introduction to the rapid decay property.
\newblock In {\em Around {L}anglands correspondences}, volume 691 of {\em
  Contemp. Math.}, pages 53--72. Amer. Math. Soc., Providence, RI, 2017.

\bibitem{ConnesHenri1990Novikovconjecture}
Alain Connes and Henri Moscovici.
\newblock Cyclic cohomology, the {N}ovikov conjecture and hyperbolic groups.
\newblock {\em Topology}, 29(3):345--388, 1990.

\bibitem{Lafforgue2002BaumConnes}
Vincent Lafforgue.
\newblock {$K$}-th\'eorie bivariante pour les alg\`ebres de {B}anach et
  conjecture de {B}aum-{C}onnes.
\newblock {\em Invent. Math.}, 149(1):1--95, 2002.

\bibitem{HaagerupPisier1993operatorspace}
Uffe Haagerup and Gilles Pisier.
\newblock Bounded linear operators between {$C^*$}-algebras.
\newblock {\em Duke Math. J.}, 71(3):889--925, 1993.

\bibitem{Buchholz1999OVfree}
Artur Buchholz.
\newblock Norm of convolution by operator-valued functions on free groups.
\newblock {\em Proc. Amer. Math. Soc.}, 127(6):1671--1682, 1999.

\bibitem{RicardXu2006reducedfreeproducts}
\'Eric Ricard and Quanhua Xu.
\newblock Khintchine type inequalities for reduced free products and
  applications.
\newblock {\em J. Reine Angew. Math.}, 599:27--59, 2006.

\bibitem{CaspersKlisseLarsen2021graphproduct}
Martijn Caspers, Mario Klisse, and Nadia~S. Larsen.
\newblock Graph product {K}hintchine inequalities and {H}ecke {$\rm
  C^\ast$}-algebras: {H}aagerup inequalities, (non)simplicity, nuclearity and
  exactness.
\newblock {\em J. Funct. Anal.}, 280(1):Paper No. 108795, 41, 2021.

\bibitem{CiobanyHoltRees2013rdgraphproducts}
Laura Ciobanu, Derek~F. Holt, and Sarah Rees.
\newblock Rapid decay is preserved by graph products.
\newblock {\em J. Topol. Anal.}, 5(2):225--237, 2013.

\bibitem{ToyotaYang2026OVhyperbolic}
Ryo Toyota and Zhiyuan Yang.
\newblock An operator-valued haagerup inequality for hyperbolic groups.
\newblock 2026.

\bibitem{KempSpeicher2007strongHaagerupscalar}
Todd Kemp and Roland Speicher.
\newblock Strong {H}aagerup inequalities for free {$\mathcal{R}$}-diagonal
  elements.
\newblock {\em J. Funct. Anal.}, 251(1):141--173, 2007.

\bibitem{delaSalle2009strongHaagerup}
Mikael de~la Salle.
\newblock Strong {H}aagerup inequalities with operator coefficients.
\newblock {\em J. Funct. Anal.}, 257(12):3968--4002, 2009.

\bibitem{Schwer2023CATintroduction}
Petra Schwer.
\newblock {\em {${\rm CAT}(0)$} cube complexes---an introduction}, volume 2324
  of {\em Lecture Notes in Mathematics}.
\newblock Springer, Cham, 2023.

\bibitem{Bowditch2013coarsemedian}
Brian~H. Bowditch.
\newblock Coarse median spaces and groups.
\newblock {\em Pacific J. Math.}, 261(1):53--93, 2013.

\bibitem{SpakulaWright2017propertyA}
J\'an \v~Spakula and Nick Wright.
\newblock Coarse medians and property {A}.
\newblock {\em Algebr. Geom. Topol.}, 17(4):2481--2498, 2017.

\bibitem{NibloWrightZang2019fourpoint}
Graham~A. Niblo, Nick Wright, and Jiawen Zhang.
\newblock A four point characterisation for coarse median spaces.
\newblock {\em Groups Geom. Dyn.}, 13(3):939--980, 2019.

\bibitem{NibloWrightZang2021}
Graham~A. Niblo, Nick Wright, and Jiawen Zhang.
\newblock Coarse median algebras: the intrinsic geometry of coarse median
  spaces and their intervals.
\newblock {\em Selecta Math. (N.S.)}, 27(2):Paper No. 20, 50, 2021.

\bibitem{Bowditch2014embedding}
Brian~H. Bowditch.
\newblock Embedding median algebras in products of trees.
\newblock {\em Geom. Dedicata}, 170:157--176, 2014.

\bibitem{Bowditch2013invariancehyperbolicity}
Brian~H. Bowditch.
\newblock Invariance of coarse median spaces under relative hyperbolicity.
\newblock {\em Math. Proc. Cambridge Philos. Soc.}, 154(1):85--95, 2013.

\bibitem{Bowditch2019notes}
Brian~H. Bowditch.
\newblock Notes on coarse median spaces.
\newblock In {\em Beyond hyperbolicity}, volume 454 of {\em London Math. Soc.
  Lecture Note Ser.}, pages 3--24. Cambridge Univ. Press, Cambridge, 2019.

\bibitem{Sapir2015centroids}
Mark Sapir.
\newblock The rapid decay property and centroids in groups.
\newblock {\em J. Topol. Anal.}, 7(3):513--541, 2015.

\bibitem{BehrstockMinsky2011centroid}
Jason~A. Behrstock and Yair~N. Minsky.
\newblock Centroids and the rapid decay property in mapping class groups.
\newblock {\em J. Lond. Math. Soc. (2)}, 84(3):765--784, 2011.

\bibitem{ChatterjiRuane2005median}
I.~Chatterji and K.~Ruane.
\newblock Some geometric groups with rapid decay.
\newblock {\em Geom. Funct. Anal.}, 15(2):311--339, 2005.

\bibitem{Pisier2001}
Gilles Pisier.
\newblock {\em Similarity Problems and Completely Bounded Maps}, volume 1618 of
  {\em Lecture Notes in Mathematics}.
\newblock Springer, Berlin, Heidelberg, 2 edition, 2001.

\end{thebibliography}
\end{document}